\documentclass[11pt]{article}

\usepackage[margin=1in]{geometry}
\usepackage[T1]{fontenc}
\usepackage{lmodern}
\usepackage{mathtools}
\usepackage{amssymb}
\usepackage{amsthm}
\usepackage{booktabs}
\usepackage{enumerate}
\usepackage{longtable}
\usepackage{pdflscape}
\usepackage{bm}
\usepackage[authoryear,round]{natbib}
\usepackage[hidelinks]{hyperref}
\allowdisplaybreaks

\theoremstyle{plain}
\newtheorem{theorem}{Theorem}[section]
\newtheorem{proposition}[theorem]{Proposition}
\newtheorem{corollary}[theorem]{Corollary}
\newtheorem{lemma}[theorem]{Lemma}
\theoremstyle{definition}
\newtheorem{definition}[theorem]{Definition}
\theoremstyle{remark}
\newtheorem{remark}[theorem]{Remark}

\numberwithin{equation}{section}

\DeclareMathOperator{\aff}{aff}

\DeclareMathOperator{\diag}{diag}
\DeclareMathOperator{\adj}{adj}
\DeclareMathOperator{\Crit}{Crit}
\DeclareMathOperator{\Mode}{Mode}

\newcommand{\R}{\mathbb{R}}

\newcommand{\keywords}[1]{\par\bigskip\noindent\textbf{Keywords:} #1\par}

\title{Bounds on the Number of Modes of a\\ Gaussian Mixture Density}
\author{Hien Duy Nguyen\thanks{Corresponding author: \href{mailto:h.nguyen5@latrobe.edu.au}{h.nguyen5@latrobe.edu.au}}\\[0.5ex]
\small School of Computing, Engineering and Mathematical Sciences, La Trobe University\\
\small Bundoora, VIC 3086, Australia\\
\small Institute of Mathematics for Industry, Kyushu University, Nishi Ward, Fukuoka 819-0395, Japan}
\date{}

\begin{document}

\maketitle

\begin{abstract}
We derive explicit upper bounds for the number of nondegenerate critical points of a $k$-component Gaussian mixture density in $\R^d$, and the number of modes when the modal set is finite, together with lower bounds. By normalizing the critical-point equations by a reference component, for $k\ge2$ we get the direct Pfaffian bound
\[
U_{\mathrm{het}}(d,k)=2^{\,d+\binom{k-1}{2}}\left(d+2\min(d,k-1)+1\right)^{k-1}.
\]
For the same parameter range, an exact elimination augmented by an algebraic reciprocal variable gives the alternative bound
\[
U_{\mathrm{aug}}(d,k)=
2^{\binom{k-1}{2}}(d+1)\left((2k-1)d+2k-1\right)^{k-1}.
\]
Thus, for $k\ge2$, the best critical-point bound is their minimum. A Morse-theoretic argument improves the corresponding finite-mode upper bound to
\[
\left\lfloor \frac{\min\{U_{\mathrm{het}}(d,k),U_{\mathrm{aug}}(d,k)\}+1}{2}\right\rfloor.
\]
In the homoscedastic case, for $k\ge2$, the direct bound improves to
\[
U_{\mathrm{hom}}(d,k)=2^{\,d+\binom{k-1}{2}}\left(d+\min(d,k-1)+1\right)^{k-1},
\]
an affine-rank reduction replaces $d$ by the affine rank of the component means, and an augmented homoscedastic reduction gives the dimension-free bound
\[
U_{\mathrm{aug,hom}}(k)=2^{\binom{k-1}{2}+1}(2k)^{k-1}.
\]
On the lower-bound side, for $d,k\ge 2$ we obtain
\[
L_{\mathrm{bin}}(d,k)=k+\max_{2\le r\le \min(d,k)}\binom{k}{r},
\]
together with a padding-product family that in particular implies the linear lower bound $d+k-1$, and a seed-closure principle that packages product and padding constructions. We further give explicit bounds for the number of connected components of the critical set.
\end{abstract}

\keywords{Gaussian mixture; mode counting; critical points; fewnomial bounds; modal clustering.}

\section{Introduction}\label{sec:intro}

Finite mixtures are a standard statistical framework for modeling heterogeneity. In the model-based clustering paradigm, a finite mixture represents the population as a superposition of latent subpopulations, and the component distributions are interpreted as clusters. For continuous data, Gaussian mixtures are the dominant example in both theory and software. Standard references include \citet{McLachlanPeel2000} and \citet{ScruccaFraleyMurphyRaftery2023}.

There is, however, a parallel line of work in which clusters are defined by the geometry of a density rather than by mixture components. In modal clustering, points belong to the same cluster when they ascend to the same local maximum of the density. Closely related formulations use domains of attraction of modes or connected components of density level sets. The literature is broad: \citet{Cheng1995} analyzed mean shift as a mode-seeking procedure; \citet{LiRayLindsay2007} developed a nonparametric modal clustering procedure based on ascent to estimated modes; density-level-set clustering is reviewed by \citet{Menardi2016} and studied algorithmically by \citet{Steinwart2011}; related density-based methods include the DBSCAN algorithm of \citet{EsterEtAl1996} and the Gaussian mixture-based variational approach of \citet{LeMinhArbelForbesNguyen2026}. A standard attraction of modal clustering is that the number of clusters is then tied to the underlying probability distribution rather than imposed externally \citep{Menardi2016}.

These two viewpoints are not equivalent. A mixture component need not correspond to a density bump, and a density bump need not correspond to a single component. \citet{LiRayLindsay2007} identify this mismatch as a limitation of naive component-based clustering, while \citet{Menardi2016} notes that model-based clustering does not provide a bijection between mixture components and density bumps. Questions about how many modes a Gaussian mixture can have were studied by \citet{CarreiraPerpinanWilliams2003}, who gave an early scale-space treatment of the mode-counting problem for Gaussian mixtures. For Gaussian mixtures, the geometric source of the mismatch was clarified by \citet{RayLindsay2005}, who introduced the ridgeline manifold and showed that all critical points of a multivariate Gaussian mixture lie on a lower-dimensional manifold determined by the component means and covariances. Building on that structure, \citet{RayRen2012} proved that a two-component Gaussian mixture in $\R^d$ can have as many as $d+1$ modes, and that this upper bound is sharp.

Since \citet{RayRen2012}, the two-component planar case has been refined further. \citet{KabataMatsumotoUchidaUeki2025} classify pairs of bivariate Gaussian densities by singularity type and show that the third mode occurs only in one of the three resulting classes. In particular, when the two covariance matrices are nonproportional but codirectional in their sense, the mixture has at most two modes. Related non-Gaussian bounds were obtained by \citet{AlexandrovichHolzmannRay2013}, who extended ridgeline methods to finite mixtures of elliptical distributions and derived upper bounds for the number of modes of two-component mixtures of multivariate $t$ distributions.

For general $d$ and $k$, \citet{AmendolaEngstromHaase2019} used the theory of fewnomials \citep{Khovanskii1991} to prove that the number of nondegenerate critical points of a $k$-component Gaussian mixture in $\R^d$ is at most
\[
U_{\mathrm{AEH}}(d,k)=2^{\,d+\binom{k}{2}}(5+3d)^k,
\]
and that the same quantity bounds the number of modes whenever the modal set is finite. They also proved the lower bound
\[
L_{\mathrm{AEH}}(d,k)=\binom{k}{d}+k
\]
for integers $d,k\ge 2$. Their paper further recalled a conjecture discussed at the 2011 American Institute of Mathematics (AIM) workshop on Singular Learning Theory, namely that the maximum possible number of modes of a $d$-dimensional Gaussian mixture with $k$ components should equal
\[
m_{\mathrm{AIM}}(d,k)=\binom{d+k-1}{d}
\]
\citep[Conjecture~3.1]{AmendolaEngstromHaase2019}; see also \citet{SteeleSturmfelsWatanabe2011}.

The present manuscript progresses in the same direction as \citet{AmendolaEngstromHaase2019}, but in four complementary ways. First, we obtain sharper upper bounds by exploiting a normalization-by-one-component argument that shortens the relevant Pfaff chain from length $k$ to length $k-1$. Second, we supplement that direct $d$-dimensional bound by an augmented reduction to $k$ variables, consisting of the $k-1$ mixture-ratio variables and one algebraic reciprocal variable. This reduction is motivated by Gale-type elimination for sparse systems; cf. \citet[Chapter~6]{Sottile2011}. Third, for finite modal sets, we combine generic exponential tilting with a Morse-theoretic handle-counting argument to gain a factor of two in the passage from critical points to modes; cf. \citet[Chapters~1--3]{Knudson2015}. Fourth, we record lower-bound constructions and isolate the connected geometry of degenerate critical points. More specifically, our main contributions are as follows.
\begin{enumerate}[1.]
    \item For $k\ge2$, we prove the improved nondegenerate critical-point bound
    \[
    U_{\mathrm{het}}(d,k)=2^{\,d+\binom{k-1}{2}}\bigl(d+2\min(d,k-1)+1\bigr)^{k-1}.
    \]
    \item For $k\ge2$, reducing the critical-point equations through an augmented reciprocal variable gives the alternative bound
    \[
    U_{\mathrm{aug}}(d,k)=2^{\binom{k-1}{2}}(d+1)\bigl((2k-1)d+2k-1\bigr)^{k-1}.
    \]
    Therefore
    \[
    |\Crit_{\mathrm{nd}}(\Phi)|\le \min\{U_{\mathrm{het}}(d,k),U_{\mathrm{aug}}(d,k)\}.
    \]
    For each fixed $k$,
    \[
    \frac{U_{\mathrm{aug}}(d,k)}{U_{\mathrm{het}}(d,k)}\to 0
    \qquad\text{as } d\to\infty.
    \]
    \item For $k\ge2$ and finite modal sets, we improve the resulting mode-count bound to
    \[
    |\Mode(\Phi)|\le
    \left\lfloor
    \frac{\min\{U_{\mathrm{het}}(d,k),U_{\mathrm{aug}}(d,k)\}+1}{2}
    \right\rfloor.
    \]
    The factor of two comes from the fact that a Morse Gaussian mixture with $N$ critical points has at most $\lfloor(N+1)/2\rfloor$ local maxima.
    \item For $k\ge2$, in the homoscedastic case we improve the direct $d$-dimensional bound further to
    \[
    U_{\mathrm{hom}}(d,k)=2^{\,d+\binom{k-1}{2}}\bigl(d+\min(d,k-1)+1\bigr)^{k-1},
    \]
    sharpen it by an affine-rank reduction that replaces $d$ by $r=\dim \aff\{\mu_1,\dots,\mu_k\}$, and obtain the additional dimension-free augmented bound
    \[
    U_{\mathrm{aug,hom}}(k)=2^{\binom{k-1}{2}+1}(2k)^{k-1}.
    \]
    \item On the lower-bound side, for $d,k\ge2$ we prove
    \[
    L_{\mathrm{bin}}(d,k)=k+\max_{2\le r\le \min(d,k)}\binom{k}{r}
    \]
    by monotonicity under dimension lifting and lifted lower-dimensional constructions, and we complement this by the padding-product family
    \[
    L_{\mathrm{pp}}(d,k)=\max_{1\le s\le \min(d,\lfloor \log_2 k\rfloor)}\left(k-2^s+\left(\left\lfloor\frac{d}{s}\right\rfloor+1\right)^{s-(d\bmod s)}\left(\left\lfloor\frac{d}{s}\right\rfloor+2\right)^{d\bmod s}\right).
    \]
    In particular, the lower bound $d+k-1$ holds for all $d\ge 1$ and $k\ge 2$. We also state a seed-closure theorem that packages dimension lifting, remote padding, and Cartesian products of seed constructions.
    \item We obtain explicit upper bounds for the number of connected components of the critical set.
\end{enumerate}

The manuscript is organized as follows. Section~\ref{sec:prelim} collects notation, definitions, and the main auxiliary observations used throughout the paper. Section~\ref{sec:upper-mode-bounds} contains upper bounds for the number of modes and nondegenerate critical points, including the augmented reduction, the finite-mode Morse improvement, the homoscedastic and affine-rank refinements, and a numerical comparison of the explicit upper bounds. Section~\ref{sec:lower-mode-bounds} turns to the lower-bound constructions and numerical comparisons. Section~\ref{sec:modal-regions} studies the problem of bounding the connected components of the critical set. Section~\ref{sec:discussion} concludes with the scope of our outcomes. Appendix~\ref{app:external} records important external theorems on which our proofs depend, and Appendix~\ref{app:deferred} collects deferred proofs.
\section{Notation and technical preliminaries}\label{sec:prelim}

\subsection{Elementary definitions}\label{subsec:definitions}

Throughout, unless explicitly stated otherwise, $d\ge1$ and $k\ge1$ are integers. The main upper bounds based on a reference component assume $k\ge2$. Throughout, $\dim$ denotes Euclidean dimension, $\aff(A)$ the affine hull of a set $A\subseteq\R^d$, $\mathrm{int}(A)$ its interior, and $\partial A$ its boundary. For a map $F:\R^n\to\R^m$, $DF(x)$ denotes its Jacobian matrix at $x$; for a scalar-valued $\mathcal{C}^1$ function $f$ (that is, a continuously differentiable function), $\nabla f(x)$ denotes its gradient; and for a scalar-valued $\mathcal{C}^2$ function $f$ (that is, a twice continuously differentiable function), $D^2f(x)$ denotes its Hessian matrix. We write $\binom{a}{b}$ for binomial coefficients and $\lfloor t\rfloor$ for the floor of a real number $t$. For integers $a\ge 0$ and $b\ge 1$, the quotient and remainder are written as $\lfloor a/b\rfloor$ and $a\bmod b$, so that
\[
a=b\left\lfloor\frac{a}{b}\right\rfloor +(a\bmod b),
\qquad
0\le a\bmod b<b.
\]
All connectedness statements are with respect to the Euclidean topology.

Let
\[
\Phi(x)=\sum_{i=1}^k \alpha_i\phi_i(x), \qquad x\in \R^d,
\]
be a Gaussian mixture density, where $\alpha_i>0$, $\sum_{i=1}^k \alpha_i=1$, and
\[
\phi_i(x)=\frac{1}{(2\pi)^{d/2}\det(\Sigma_i)^{1/2}}
\exp\!\left(-\frac12(x-\mu_i)^\top\Sigma_i^{-1}(x-\mu_i)\right)
\]
with $\mu_i\in\R^d$ and each $\Sigma_i$ positive definite. We write
\[
\Crit(\Phi)=\{x\in\R^d:\nabla \Phi(x)=0\}
\]
for the critical set. To handle possible connected sets of local maxima, we define the modal subset of $\Crit(\Phi)$.

\begin{definition}
The modal set of $\Phi$ is
\[
\Mode(\Phi)=\left\{x\in\R^d:\exists r>0\ \forall y\in\R^d\ \left(\|y-x\|<r\Rightarrow \Phi(y)\le \Phi(x)\right)\right\}.
\]
Its Euclidean connected components are called the connected modal regions of $\Phi$.
\end{definition}

We will also require the corresponding nondegenerate subsets of the critical and modal sets.

\begin{definition}
A point $x_\ast\in\R^d$ is a mode of $\Phi$ if it is a local maximum of $\Phi$. A critical point $x_\ast\in\Crit(\Phi)$ is nondegenerate if the Hessian $D^2\Phi(x_\ast)$ is nonsingular. We write
\[
\Crit_{\mathrm{nd}}(\Phi)=\left\{x\in\Crit(\Phi): \det D^2\Phi(x)\neq 0\right\}
\]
for the set of nondegenerate critical points, and
\[
\Mode_{\mathrm{nd}}(\Phi)=\Mode(\Phi)\cap \Crit_{\mathrm{nd}}(\Phi)
\]
for the set of nondegenerate modes.
\end{definition}

\begin{remark}\label{rem:mode-nd}
The set $\Mode(\Phi)$ records all local-maximal points of $\Phi$, including degenerate maxima and possible connected modal sets. The set $\Crit_{\mathrm{nd}}(\Phi)$ consists of the critical points with nonsingular Hessian, and $\Mode_{\mathrm{nd}}(\Phi)$ is the subset of those that are local maxima. In particular, every point of $\Mode_{\mathrm{nd}}(\Phi)$ is an isolated local maximum. The sets $\Mode(\Phi)$ and $\Mode_{\mathrm{nd}}(\Phi)$ agree when every mode is nondegenerate.
\end{remark}

\begin{remark}\label{rem:m-vs-Mode}
We retain the notation $m(d,k)$ for the supremum of $|\Mode(\Phi)|$ over $k$-component Gaussian mixtures in $\R^d$ whose modal set is finite. The homoscedastic analogue is denoted $m_{\mathrm{hom}}(d,k)$. Thus these are finite-modal-set counting quantities, while $\Mode(\Phi)$ denotes the modal set of one fixed mixture density $\Phi$ and may include connected modal regions.
\end{remark}

\subsection{A compact containment for the critical set}\label{subsec:compact-critical}

The next basic lemma is standard but useful. It shows that all critical points lie in a compact set depending only on the component parameters.

\begin{lemma}\label{lem:compact-critical}
For every Gaussian mixture density $\Phi$,
\[
\Crit(\Phi)\subseteq \mathcal M=\left\{\left(\sum_{i=1}^k w_i\Sigma_i^{-1}\right)^{-1}
\left(\sum_{i=1}^k w_i\Sigma_i^{-1}\mu_i\right): w=(w_1,\dots,w_k)\in \Delta_k\right\},
\]
where $\Delta_k=\{w\in[0,1]^k:\sum_{i=1}^k w_i=1\}$ is the probability simplex. In particular, $\Crit(\Phi)$ is compact.
\end{lemma}

\begin{proof}
If $x\in\Crit(\Phi)$, set
\[
w_i=\frac{\alpha_i\phi_i(x)}{\Phi(x)},\qquad i=1,\dots,k.
\]
Then $w\in\Delta_k$ and the critical-point equation becomes
\[
\sum_{i=1}^k w_i\Sigma_i^{-1}(\mu_i-x)=0.
\]
Rearranging gives the displayed representation of $x$. Since $\Delta_k$ is compact and the displayed map is continuous, $\mathcal M$ is compact. The set $\Crit(\Phi)$ is closed because it is the zero set of the continuous map $\nabla\Phi$.
\end{proof}

\begin{lemma}\label{lem:modal-critical}
For every Gaussian mixture density $\Phi$,
\[
\Mode(\Phi)\subseteq \Crit(\Phi).
\]
In particular, the modal set is contained in a compact set.
\end{lemma}

\begin{proof}
A local maximum of a $\mathcal{C}^1$ function has vanishing gradient, so every point of $\Mode(\Phi)$ is critical. The containment in a compact set then follows from Lemma~\ref{lem:compact-critical}.
\end{proof}

\begin{proposition}\label{prop:nonsingleton-critical-components}
Every non-singleton connected component of $\Crit(\Phi)$ lies in the degenerate locus
\[
\{x\in\R^d:\nabla\Phi(x)=0,\ \det D^2\Phi(x)=0\}.
\]
\end{proposition}

\begin{proof}
Let $K$ be a connected component of $\Crit(\Phi)$ containing more than one point, and let $x\in K$. If $D^2\Phi(x)$ were nonsingular, then $x$ would be an isolated critical point by the inverse function theorem applied to $\nabla\Phi$. The connected component of $x$ in $\Crit(\Phi)$ would then be the singleton $\{x\}$, contrary to the hypothesis on $K$. Hence every point of $K$ is degenerate.
\end{proof}

The next observation explains why the finiteness of the number of modes for a Gaussian-mixture density is fundamentally a degeneracy problem.

\begin{proposition}\label{prop:degeneracy-obstruction}
If $|\Mode(\Phi)|=\infty$, then $\Phi$ has a degenerate critical point.
\end{proposition}

\begin{proof}
Every mode is a critical point, so an infinite modal set is an infinite subset of the compact set $\Crit(\Phi)$ from Lemma~\ref{lem:compact-critical}. Hence it has an accumulation point $x_\ast\in\Crit(\Phi)$. A nondegenerate critical point is isolated by the inverse function theorem applied to $\nabla\Phi$. Therefore $x_\ast$ must be degenerate.
\end{proof}

\subsection{Morse theory}\label{subsec:morse-terminology}

We use only standard Morse-theoretic facts regarding handle decompositions. A $\mathcal{C}^2$ function on an open subset of $\R^d$ is Morse if all its critical points are nondegenerate. If $p$ is a nondegenerate critical point, its Morse index is the number of negative eigenvalues of the Hessian at $p$, counted with multiplicity. Thus a local maximum has Morse index $d$. A value $a\in\R$ is a regular value if the level set $\{h=a\}$ contains no critical point. We write $\pi_0(A)$ for the set of path-connected components of a topological space $A$.

Let $\mathbb D^j=\{u\in\R^j:\|u\|\le1\}$ denote the closed unit $j$-disk, with $\mathbb D^0$ a point and $\partial\mathbb D^0=\varnothing$. We say that a $j$-handle in dimension $d$ is a copy of $\mathbb D^j\times \mathbb D^{d-j}$ attached along $\partial \mathbb D^j\times \mathbb D^{d-j}$, where its index is $j$. We use the standard handle-attachment theorem for Morse functions in the form recorded in Theorem~\ref{thm:morse-handle}. In particular, when a critical value of a Morse function on a $d$-manifold is crossed upward in sublevel sets, one attaches handles whose indices are the Morse indices of the critical points at that level. That is, if the level contains critical points of Morse indices $\lambda_1,\dots,\lambda_s$, then the upper sublevel set is diffeomorphic, after smoothing corners, to the lower sublevel set with $\lambda_i$-handles attached for $i=1,\dots,s$. Equivalently, when a critical value of $h$ is crossed downward in the superlevel sets $\{h\ge t\}$, one applies the sublevel-set theorem to $-h$, and a critical point of Morse index $\lambda$ for $h$ contributes a $(d-\lambda)$-handle. These definitions and the handle-decomposition theorem are standard; see \citet[Chapters~1--3]{Knudson2015}.

\subsection{The normalized critical system}\label{subsec:normalized-system}

A Pfaff chain of length $q$ and degree $M$ on an open set $U\subseteq\R^n$ is a tuple of smooth functions $y_1,\dots,y_q$ such that for every $1\le i\le q$ and $1\le j\le n$,
\[
\frac{\partial y_i}{\partial x_j}(x)=P_{ij}\bigl(x,y_1(x),\dots,y_i(x)\bigr)
\]
with $P_{ij}$ a polynomial of total degree at most $M$. A Pfaffian zero set is a set cut out by finitely many equations
\[
F_1(x)=\cdots=F_\ell(x)=0
\]
where the $F_j$ are Pfaffian functions with respect to a common Pfaff chain.

Fix the $k$th component as a reference and write $q=k-1$. For $i=1,\dots,q$, define the positive ratios
\[
\rho_i(x)=\frac{\alpha_i\phi_i(x)}{\alpha_k\phi_k(x)}.
\]
Each $\rho_i$ is a positive constant times the exponential of a quadratic polynomial; explicitly there exist constants $\beta_i>0$ and quadratic polynomials $q_i(x)$ such that
\[
\rho_i(x)=\beta_i e^{q_i(x)}.
\]
Hence
\[
\frac{\partial \rho_i}{\partial x_j}(x)=\ell_{ij}(x)\rho_i(x),
\]
where each $\ell_{ij}$ is affine linear. Thus $\rho_1,\dots,\rho_q$ form a Pfaff chain of degree $2$ in the standard sense fixed above. If all covariance matrices are equal, then each $q_i$ is affine rather than quadratic and the chain degree drops to $1$.

Now set
\[
b_i(x)=-\Sigma_i^{-1}(x-\mu_i),\qquad i=1,\dots,k.
\]
Since
\[
\nabla \Phi(x)=\sum_{i=1}^k \alpha_i\phi_i(x)b_i(x),
\]
dividing by the positive factor $\alpha_k\phi_k(x)$ yields the equivalent system
\begin{equation}\label{eq:normalized-system}
b_k(x)+\sum_{i=1}^{q}\rho_i(x)b_i(x)=0.
\end{equation}
For $r=1,\dots,d$, let $b_{i,r}$ denote the $r$th coordinate of $b_i$ and define
\[
Q_r(x,y)=b_{k,r}(x)+\sum_{i=1}^{q} y_i b_{i,r}(x).
\]
Then each $Q_r$ is a polynomial of total degree at most $2$ in $(x,y)$, and the critical-point equations become
\[
Q_1(x,\rho(x))=\cdots=Q_d(x,\rho(x))=0.
\]
This is the Pfaffian system to which we apply Theorem~\ref{thm:Khov}.

\section{Upper bounds on the number of modes}\label{sec:upper-mode-bounds}

We first collect bounds for $\Crit_{\mathrm{nd}}(\Phi)$, and then convert them into sharper finite-mode bounds.

\subsection{A normalized Pfaffian bound}\label{subsec:general-bound}

The first quantitative improvement over \citet{AmendolaEngstromHaase2019} comes from the normalization in Section~\ref{subsec:normalized-system}. Their fewnomial formulation uses $k$ exponential terms, while the normalized system uses only the $k-1$ ratios $\rho_1,\dots,\rho_{k-1}$.

\begin{theorem}\label{thm:main-general}
Let $\Phi$ be a Gaussian mixture density with $k\ge 2$ components in $\R^d$. Then
\[
|\Crit_{\mathrm{nd}}(\Phi)|\le U_{\mathrm{het}}(d,k)=2^{\,d+\binom{k-1}{2}}\bigl(d+2\min(d,k-1)+1\bigr)^{k-1}.
\]
\end{theorem}

\begin{proof}
The functions $\rho_1,\dots,\rho_q$ form a Pfaff chain of length $q=k-1$ and degree $2$ on $\R^d$, because for every $i$ and $j$ one has
\[
\frac{\partial \rho_i}{\partial x_j}(x)=\ell_{ij}(x)\rho_i(x)
\]
with $\ell_{ij}$ affine linear. The equations
\[
Q_1(x,\rho(x))=\cdots=Q_d(x,\rho(x))=0
\]
therefore form a Pfaffian system in the $d$ variables $x_1,\dots,x_d$, with $n=d$ and polynomial degrees $p_r=2$ for all $r$.

Applying Theorem~\ref{thm:Khov} with
\[
n=d,\qquad q=k-1,\qquad M=2,\qquad p_1=\cdots=p_d=2,
\]
we obtain at most
\[
2^{q(q-1)/2}\,2^d\,L^q
\]
nondegenerate roots, where
\[
L=\sum_{r=1}^d(2-1)+\min(d,k-1)\cdot 2+1=d+2\min(d,k-1)+1.
\]
Thus the normalized system has at most
\[
2^{\binom{k-1}{2}}2^d\bigl(d+2\min(d,k-1)+1\bigr)^{k-1}=U_{\mathrm{het}}(d,k)
\]
nondegenerate roots.

It remains to identify these roots with the nondegenerate critical points of $\Phi$. Define
\[
G(x)=b_k(x)+\sum_{i=1}^{q}\rho_i(x)b_i(x)=\frac{1}{\alpha_k\phi_k(x)}\nabla\Phi(x).
\]
Hence $G(x)=0$ if and only if $x\in\Crit(\Phi)$. If $x_\ast\in\Crit(\Phi)$, then $\nabla\Phi(x_\ast)=0$, so differentiating the identity $G=(\alpha_k\phi_k)^{-1}\nabla\Phi$ gives
\[
DG(x_\ast)=\frac{1}{\alpha_k\phi_k(x_\ast)}D^2\Phi(x_\ast).
\]
The scalar factor is positive, so $DG(x_\ast)$ is nonsingular if and only if $D^2\Phi(x_\ast)$ is nonsingular. Therefore nondegenerate roots of the normalized system are exactly the points of $\Crit_{\mathrm{nd}}(\Phi)$, and the bound follows.
\end{proof}

\subsection{Finite modal sets}\label{subsec:morse-factor}

The preceding theorem bounds nondegenerate critical points. The argument in this subsection improves the standard perturbation passage from nondegenerate critical points to modes when the modal set is finite.

\begin{lemma}\label{lem:low-superlevel-connected}
Let $h$ be a Morse Gaussian mixture density on $\R^d$. Then there exists a regular value $\tau>0$, below all critical values of $h$, such that
\[
S_\tau=\{x\in\R^d:h(x)\ge \tau\}
\]
is compact, connected, and contains every critical point of $h$.
\end{lemma}

\begin{proof}
See Appendix~\ref{subsec:low-superlevel-proof}.
\end{proof}

\begin{lemma}[Connected components under handle attachment]\label{lem:handle-components}
Let $A$ be a $d$-manifold with boundary, possibly empty, with finitely many path-connected components, and let $B$ be obtained from $A$ by attaching finitely many handles. Let $h_0$ and $h_1$ be the numbers of $0$-handles and $1$-handles attached. Then
\[
|\pi_0(B)|\ge |\pi_0(A)|+h_0-h_1.
\]
\end{lemma}

\begin{proof}
It is enough to attach the handles one at a time. Suppose that $A'$ is the space already constructed and that $A''$ is obtained from $A'$ by attaching one $j$-handle
\[
H=\mathbb D^j\times\mathbb D^{d-j}
\]
along
\[
E=\partial\mathbb D^j\times\mathbb D^{d-j}.
\]
The handle $H$ is path-connected for every $j$.

If $j=0$, then $E=\varnothing$ by the convention above. Thus the handle is added disjointly to $A'$, and it contributes exactly one new path-connected component. Hence $|\pi_0(A'')|=|\pi_0(A')|+1$ in this case.

If $j=1$, then $E=\partial\mathbb D^1\times\mathbb D^{d-1}$ is the disjoint union of two $(d-1)$-disks. If the images of these two disks lie in the same path-connected component of $A'$, then the connected handle $H$ is attached to that component and no two components of $A'$ are joined; the number of components is unchanged. If the two images lie in distinct path-connected components of $A'$, then a path inside $H$ joins those two components after the attachment. In that case exactly those two components merge, and the number of components decreases by one. Thus a $1$-handle can decrease the number of components by at most one.

If $j\ge2$, then $E=\partial\mathbb D^j\times\mathbb D^{d-j}$ is nonempty and path-connected. Its image is therefore contained in a single path-connected component of $A'$. Since $H$ is path-connected and meets that component along a nonempty set, attaching $H$ neither creates a new component nor joins two old components. Hence the number of components is unchanged for handles of index at least two.

Summing these three alternatives over all attached handles gives
\[
|\pi_0(B)|\ge |\pi_0(A)|+h_0-h_1.
\]
\end{proof}

\noindent\textbf{Remark.} This component-counting argument is standard in Morse theory. For a closely related use, see \citet[Theorem~3.22]{Knudson2015}, where a connected manifold is analyzed by first considering the components created by index-$0$ critical points and then observing that handles of index at least two cannot join distinct components because their attaching spheres are connected. Another example appears in \citet[Theorem~2.37]{Nicolaescu2011}, in the proof of the Heegaard-decomposition theorem.

\begin{proposition}\label{prop:morse-half}
Let $h$ be a Morse Gaussian mixture density on $\R^d$. If $N$ is the total number of critical points of $h$ and $M$ is the number of local maxima of $h$, then
\[
M\le \left\lfloor\frac{N+1}{2}\right\rfloor.
\]
\end{proposition}

\begin{proof}
Choose $\tau$ as in Lemma~\ref{lem:low-superlevel-connected}. For $t>0$, write
\[
S_t=\{x:h(x)\ge t\}.
\]
For $t$ larger than $\max h$, the set $S_t$ is empty. Since $S_\tau$ is compact and $h$ is Morse, the critical points in $S_\tau$ are finite, and hence there are only finitely many critical values in $[\tau,\max h]$. Choose regular values separating these critical values. On each band between two adjacent regular values the set
\[
\{x:a\le -h(x)\le b\}
\]
is compact because it is contained in $S_\tau$. Thus Theorem~\ref{thm:morse-handle} applies to $-h$ on each such band, and the number of path-connected components of $S_t$ can change only when $t$ crosses a critical value.

By the superlevel-set form of Theorem~\ref{thm:morse-handle}, if $p$ is a critical point of Morse index $\lambda$ for $h$, then crossing its critical value downward attaches a handle of index $d-\lambda$ to the current superlevel set. Thus a local maximum of $h$ has Morse index $d$ and contributes a $0$-handle, while a critical point of Morse index $d-1$ contributes a $1$-handle. All remaining critical points contribute handles of index at least two.

Let $C_{d-1}$ denote the number of critical points of Morse index $d-1$. Starting from the empty superlevel set and crossing all critical values down to $\tau$, Lemma~\ref{lem:handle-components} gives
\[
|\pi_0(S_\tau)|\ge M-C_{d-1}.
\]
Lemma~\ref{lem:low-superlevel-connected} gives $|\pi_0(S_\tau)|=1$, so
\[
C_{d-1}\ge M-1.
\]
Since every local maximum and every index-$(d-1)$ critical point is included among the $N$ critical points,
\[
N\ge M+C_{d-1}\ge 2M-1.
\]
Rearranging gives the claimed inequality.
\end{proof}

\begin{lemma}[Exponential tilting]\label{lem:generic-tilt}
Let $\Phi$ be a Gaussian mixture density. Suppose that $K_1,\dots,K_M$ are pairwise disjoint compact neighborhoods with $p_i\in\mathrm{int}(K_i)$ and
\[
\Phi(p_i)>\max_{x\in\partial K_i}\Phi(x),
\qquad i=1,\dots,M.
\]
For every neighborhood of $0$ in $\R^d$, there exists $c$ in that neighborhood such that
\[
H_c(x)=e^{c\cdot x}\Phi(x)
\]
is a positive scalar multiple of a Gaussian mixture density with the same number of components, has only nondegenerate critical points, and has at least one local maximum in each $K_i$. If $\Phi$ is homoscedastic, then the tilted mixture is homoscedastic with the same common covariance and the same affine rank of the component means.
\end{lemma}

\begin{proof}
See Appendix~\ref{subsec:generic-tilt-proof}.
\end{proof}

\begin{remark}\label{rem:generic-tilt-consequences}
The normalized density proportional to $H_c$ has the same critical points and local maxima as $H_c$. Since its critical set is compact by Lemma~\ref{lem:compact-critical} and all critical points are nondegenerate, this critical set is finite. Hence the normalized tilted mixture has finite modal set and at least $M$ nondegenerate modes. If $\Phi$ has finite modal set, then the neighborhoods $K_i$ may be chosen around all points of $\Mode(\Phi)$, and the normalized tilted mixture has at least $|\Mode(\Phi)|$ nondegenerate modes.
\end{remark}

\begin{proposition}[Finite-mode transfer]\label{prop:finite-mode-transfer}
Let $\mathcal C$ be a class of Gaussian mixture densities such that, for every $\Phi\in\mathcal C$ and every sufficiently small $c\in\R^d$, the function $e^{c\cdot x}\Phi(x)$ is a positive scalar multiple of a density in $\mathcal C$. Suppose that every mixture in $\mathcal C$ has at most $U$ nondegenerate critical points. If $\Phi\in\mathcal C$ has finite modal set, then
\[
|\Mode(\Phi)|\le \left\lfloor\frac{U+1}{2}\right\rfloor.
\]
\end{proposition}

\begin{proof}
Since $\Mode(\Phi)$ is finite, each of its points is an isolated strict local maximum. Choose pairwise disjoint compact neighborhoods of these modes on which the strict boundary inequalities in Lemma~\ref{lem:generic-tilt} hold, and choose the corresponding exponential tilt $H_c$. Let $\widehat H_c$ be the density in $\mathcal C$ proportional to $H_c$. The functions $H_c$ and $\widehat H_c$ have the same critical points and local maxima, and all these critical points are nondegenerate. Hence the number $N_c$ of critical points of $\widehat H_c$ is at most $U$. By Proposition~\ref{prop:morse-half}, the number of local maxima of $\widehat H_c$ is at most $\lfloor(N_c+1)/2\rfloor$, and this is at most $\lfloor(U+1)/2\rfloor$. Since Lemma~\ref{lem:generic-tilt} gives at least $|\Mode(\Phi)|$ local maxima of $H_c$, and hence of $\widehat H_c$, the result follows.
\end{proof}

Combining Theorem~\ref{thm:main-general} with Proposition~\ref{prop:finite-mode-transfer} gives the direct finite-mode consequence
\[
|\Mode(\Phi)|\le \left\lfloor\frac{U_{\mathrm{het}}(d,k)+1}{2}\right\rfloor
\]
whenever $\Mode(\Phi)$ is finite. We will combine this argument with sharper bounds below.

\subsection[An augmented reduction approach]{An augmented reduction approach}\label{subsec:reduced-bound}

The normalization by one reference component can be pushed further: the variable $x$ can be eliminated. The key point is to keep the reciprocal of a determinant as an ordinary algebraic variable rather than as an additional Pfaff-chain element.

Fix again the $k$th component as reference and write $m=k-1$ and $A_i=\Sigma_i^{-1}$. For $y=(y_1,\dots,y_m)\in (0,\infty)^m$, define
\[
M(y)=A_k+\sum_{i=1}^m y_iA_i,
\qquad
\nu(y)=A_k\mu_k+\sum_{i=1}^m y_iA_i\mu_i,
\qquad
X(y)=M(y)^{-1}\nu(y).
\]
For $i=1,\dots,m$, set
\[
\beta_i=\frac{\alpha_i\det(\Sigma_k)^{1/2}}{\alpha_k\det(\Sigma_i)^{1/2}}
\]
and
\[
q_i(x)=-\frac12(x-\mu_i)^\top A_i(x-\mu_i)+\frac12(x-\mu_k)^\top A_k(x-\mu_k).
\]
Then
\[
\rho_i(x)=\frac{\alpha_i\phi_i(x)}{\alpha_k\phi_k(x)}=\beta_i e^{q_i(x)}.
\]
The usual reduced map $R=(R_1,\dots,R_m):(0,\infty)^m\to \R^m$ is
\[
R_i(y)=y_i-\beta_i e^{q_i(X(y))},\qquad i=1,\dots,m.
\]

\begin{theorem}[Positive-orthant reduction]\label{thm:reduction}
The map $X:(0,\infty)^m\to\R^d$ restricts to a bijection between the positive solution set of the reduced system
\[
R_1(y)=\cdots=R_m(y)=0
\]
and the critical set $\Crit(\Phi)$ of the Gaussian mixture density $\Phi$. Explicitly, if $x\in\Crit(\Phi)$, then the corresponding reduced coordinates are uniquely given by
\[
y_i=\rho_i(x)=\frac{\alpha_i\phi_i(x)}{\alpha_k\phi_k(x)},\qquad i=1,\dots,m,
\]
and conversely every positive root $y$ of $R$ gives the critical point $x=X(y)$.
\end{theorem}

\begin{proof}
Suppose first that $x\in\Crit(\Phi)$ and define $y_i=\rho_i(x)$ for $i=1,\dots,m$. Then $y_i>0$. The normalized critical equations from \eqref{eq:normalized-system} read
\[
A_k(\mu_k-x)+\sum_{i=1}^m y_iA_i(\mu_i-x)=0.
\]
Rearranging gives
\[
\left(A_k+\sum_{i=1}^m y_iA_i\right)x=A_k\mu_k+\sum_{i=1}^m y_iA_i\mu_i,
\]
so $x=X(y)$. Since $y_i=\rho_i(x)=\beta_i e^{q_i(x)}$ and $x=X(y)$, we obtain $R_i(y)=0$ for all $i$.

Conversely, let $y\in(0,\infty)^m$ satisfy $R(y)=0$ and put $x=X(y)$. By definition of $X$,
\[
A_k(\mu_k-x)+\sum_{i=1}^m y_iA_i(\mu_i-x)=0.
\]
Since $R_i(y)=0$, we also have $y_i=\beta_i e^{q_i(x)}=\rho_i(x)$ for each $i$. Multiplying the last vector equation by $\alpha_k\phi_k(x)>0$ gives
\[
\sum_{i=1}^k \alpha_i\phi_i(x)A_i(\mu_i-x)=\nabla\Phi(x)=0.
\]
Hence $x\in\Crit(\Phi)$.

The two constructions are inverse to one another. Starting from $x\in\Crit(\Phi)$, forming $y_i=\rho_i(x)$ and then solving $M(y)x=\nu(y)$ recovers the original point $x$ because $M(y)$ is positive definite. Starting from a positive root $y$ of $R$, forming $x=X(y)$, and then evaluating the ratios gives $\rho_i(x)=y_i$ because $R_i(y)=0$.
\end{proof}

\begin{proposition}\label{prop:reduced-jacobian-hessian}
Let $y\in(0,\infty)^m$ satisfy $R(y)=0$, and set $x=X(y)$. Then $DR(y)$ is nonsingular if and only if $D^2\Phi(x)$ is nonsingular.
\end{proposition}

\begin{proof}
Define
\[
B(y,x)=A_k(\mu_k-x)+\sum_{i=1}^m y_iA_i(\mu_i-x).
\]
Then $B(y,X(y))=0$ for every $y$, because this identity is exactly $M(y)X(y)=\nu(y)$. At a root of the reduced system one has $y_i=\rho_i(x)$, and the normalized critical vector is
\[
G(x)=A_k(\mu_k-x)+\sum_{i=1}^m\rho_i(x)A_i(\mu_i-x)=B(\rho(x),x).
\]
Thus $G(x)=0$ if and only if $x\in\Crit(\Phi)$, and at such a point
\[
DG(x)=\frac{1}{\alpha_k\phi_k(x)}D^2\Phi(x),
\]
as in the proof of Theorem~\ref{thm:main-general}. Hence $DG(x)$ is nonsingular if and only if $D^2\Phi(x)$ is nonsingular.

It remains to compare $DG(x)$ with $DR(y)$. Put
\[
C=D\rho(x)\in\R^{m\times d},
\qquad
J=DX(y)\in\R^{d\times m},
\qquad
P=D_yB(y,x)\in\R^{d\times m}.
\]
Since $D_xB(y,x)=-M(y)$, differentiating $B(y,X(y))=0$ gives
\[
P-M(y)J=0,
\qquad\text{so}\qquad P=M(y)J.
\]
Also,
\[
DR(y)=I_m-CJ,
\]
while
\[
DG(x)=D_xB(y,x)+D_yB(y,x)D\rho(x)=-M(y)+PC
       =-M(y)(I_d-JC).
\]
The matrix $M(y)$ is positive definite and hence invertible. Sylvester's determinant identity gives
\[
\det(I_m-CJ)=\det(I_d-JC).
\]
Therefore $DR(y)$ is nonsingular if and only if $DG(x)$ is nonsingular, and the preceding paragraph identifies this with nonsingularity of $D^2\Phi(x)$.
\end{proof}

Set
\[
D(y)=\det M(y),
\qquad
N(y)=\adj(M(y))\nu(y).
\]
Each entry of $M(y)$ is affine in $y$, so $D(y)$ has degree at most $d$. Each coordinate of $N(y)$ has degree at most $d$, because $\adj(M(y))$ has entries of degree at most $d-1$ and $\nu(y)$ is affine. Moreover $D(y)>0$ on $(0,\infty)^m$, and
\[
X(y)=\frac{N(y)}{D(y)}.
\]
We now augment the reduced $y$-system by adding one ordinary algebraic reciprocal variable $z$, constrained by $zD(y)=1$. Define
\[
\eta_i(y,z)=\exp\{q_i(zN(y))\},
\qquad i=1,\dots,m.
\]
Consider the square system in the $k$ variables $(y,z)\in(0,\infty)^m\times\R$:
\begin{equation}\label{eq:aug-system}
E_0(y,z)=zD(y)-1=0,
\qquad
E_i(y,z)=y_i-\beta_i\eta_i(y,z)=0,\quad i=1,\dots,m.
\end{equation}

\begin{lemma}[Bijection and nondegeneracy of the augmented system]\label{lem:aug-bijection}
Positive roots of the reduced system $R(y)=0$ are in bijection with roots of \eqref{eq:aug-system} in $(0,\infty)^m\times\R$, via
\[
y\longmapsto (y,D(y)^{-1}).
\]
Under this bijection, a root of \eqref{eq:aug-system} is nondegenerate if and only if the corresponding point of $\Crit(\Phi)$ is nondegenerate.
\end{lemma}

\begin{proof}
If $y$ solves $R(y)=0$, then $(y,D(y)^{-1})$ solves \eqref{eq:aug-system}, because $D(y)^{-1}N(y)=X(y)$. Conversely, $E_0(y,z)=0$ implies $z=D(y)^{-1}$, and then the equations $E_i=0$ reduce exactly to $R_i(y)=0$.

For nondegeneracy, write $h(y,z)=(E_1(y,z),\dots,E_m(y,z))$. Since
\[
\frac{\partial E_0}{\partial z}(y,z)=D(y)>0
\]
at every root, the Schur complement of the $E_0$-row and $z$-column shows that the Jacobian $DE(y,z)$ is nonsingular if and only if
\[
D_y h(y,s(y))
\]
is nonsingular, where $s(y)=D(y)^{-1}$. But $h(y,s(y))=R(y)$, so this Schur complement is $DR(y)$. Proposition~\ref{prop:reduced-jacobian-hessian} then gives the equivalence with nonsingularity of $D^2\Phi(X(y))$.
\end{proof}

\begin{lemma}\label{lem:aug-pfaff}
The functions $\eta_1,\dots,\eta_m$ form a Pfaff chain of length $m=k-1$ and degree at most $2d+2$ on $(0,\infty)^m\times\R$. With respect to this chain, the equations \eqref{eq:aug-system} have polynomial degrees
\[
p_0=d+1,
\qquad
p_1=\cdots=p_m=1.
\]
\end{lemma}

\begin{proof}
The polynomial $q_i$ has degree at most $2$, and every coordinate of $zN(y)$ has total degree at most $d+1$ in $(y,z)$. Hence
\[
P_i(y,z)=q_i(zN(y))
\]
has total degree at most $2d+2$. Since $\eta_i=e^{P_i}$, for any coordinate variable $w\in\{y_1,\dots,y_m,z\}$,
\[
\frac{\partial \eta_i}{\partial w}(y,z)
=\eta_i(y,z)\frac{\partial P_i}{\partial w}(y,z).
\]
The derivative $\partial P_i/\partial w$ has degree at most $2d+1$. Therefore $\partial\eta_i/\partial w$ is a polynomial in $(y,z,\eta_i)$ of total degree at most $2d+2$. The derivative of $\eta_i$ does not depend on any later chain element, so ordering the chain as $\eta_1,\dots,\eta_m$ proves the Pfaff-chain assertion.

The first equation $zD(y)-1=0$ has degree at most $d+1$, while each equation $y_i-\beta_i\eta_i=0$ has degree $1$ in $(y,z,\eta)$.
\end{proof}

\begin{theorem}[Augmented reduced Pfaffian bound]\label{thm:augmented-bound}
For every Gaussian mixture with $k\ge 2$ components in $\R^d$,
\[
|\Crit_{\mathrm{nd}}(\Phi)|\le U_{\mathrm{aug}}(d,k)
=2^{\binom{k-1}{2}}(d+1)\bigl((2k-1)d+2k-1\bigr)^{k-1}.
\]
\end{theorem}

\begin{proof}
Apply Theorem~\ref{thm:Khov} to \eqref{eq:aug-system}. The number of variables is $n=k$, the chain length is $q=k-1$, and Lemma~\ref{lem:aug-pfaff} gives chain degree $M=2d+2$ and polynomial degrees
\[
p_0=d+1,\qquad p_1=\cdots=p_m=1.
\]
Therefore
\[
\prod_{r=0}^m p_r=d+1
\]
and
\begin{align*}
L
&=\sum_{r=0}^m(p_r-1)+\min(k,k-1)(2d+2)+1 \\
&=d+(k-1)(2d+2)+1\\
&=(2k-1)d+2k-1.
\end{align*}
Theorem~\ref{thm:Khov} gives at most
\[
2^{\binom{k-1}{2}}(d+1)\bigl((2k-1)d+2k-1\bigr)^{k-1}
\]
nondegenerate roots of \eqref{eq:aug-system}. Lemma~\ref{lem:aug-bijection} identifies those roots with $\Crit_{\mathrm{nd}}(\Phi)$.
\end{proof}

\begin{corollary}\label{cor:best-heteroscedastic-bound}
For every Gaussian mixture with $k\ge2$ components in $\R^d$,
\[
|\Crit_{\mathrm{nd}}(\Phi)|
\le
U_{\mathrm{best}}(d,k)=\min\{U_{\mathrm{het}}(d,k),U_{\mathrm{aug}}(d,k)\}.
\]
If the modal set is finite, then
\[
|\Mode(\Phi)|\le
\left\lfloor\frac{U_{\mathrm{best}}(d,k)+1}{2}\right\rfloor.
\]
\end{corollary}

\begin{proof}
The nondegenerate-critical bound follows by combining Theorems~\ref{thm:main-general} and \ref{thm:augmented-bound}. The finite-mode statement follows from Proposition~\ref{prop:finite-mode-transfer}.
\end{proof}

\begin{remark}
For each fixed $k$, the augmented reduced bound is polynomial in $d$:
\[
U_{\mathrm{aug}}(d,k)=O_k(d^k).
\]
By contrast, $U_{\mathrm{het}}(d,k)=O_k(2^d d^{k-1})$. Hence, for each fixed $k$,
\[
\frac{U_{\mathrm{aug}}(d,k)}{U_{\mathrm{het}}(d,k)}\to 0
\qquad (d\to\infty).
\]
There is nevertheless no uniform dominance between $U_{\mathrm{het}}$ and $U_{\mathrm{aug}}$ over all pairs $(d,k)$; the pointwise best nondegenerate-critical bound is their minimum.
\end{remark}

\subsection{Homoscedastic refinements}\label{subsec:homoscedastic}

When the covariance matrices coincide, the normalized ratios are exponentials of affine rather than quadratic functions, so the direct $d$-dimensional Pfaff chain degree falls from $2$ to $1$.

\begin{corollary}[Homoscedastic specialization]\label{cor:hom-special}
Assume that $\Sigma_1=\cdots=\Sigma_k=\Sigma$. Then
\[
|\Mode_{\mathrm{nd}}(\Phi)|\le |\Crit_{\mathrm{nd}}(\Phi)|\le U_{\mathrm{hom}}(d,k)=2^{\,d+\binom{k-1}{2}}\bigl(d+\min(d,k-1)+1\bigr)^{k-1}.
\]
If the modal set is finite, then
\[
|\Mode(\Phi)|\le \left\lfloor\frac{U_{\mathrm{hom}}(d,k)+1}{2}\right\rfloor.
\]
\end{corollary}

\begin{proof}
In the homoscedastic case, each $\rho_i$ is the exponential of an affine function, so the Pfaff chain degree becomes $M=1$. Applying Theorem~\ref{thm:Khov} exactly as in the proof of Theorem~\ref{thm:main-general}, but with $M=1$, gives
\[
L=d+\min(d,k-1)+1
\]
and therefore the displayed nondegenerate-critical bound. For the finite-mode statement, apply Proposition~\ref{prop:finite-mode-transfer}; exponential tilting preserves the homoscedastic class by Lemma~\ref{lem:generic-tilt}.
\end{proof}

The homoscedastic case admits a more structural improvement. In that setting every critical point lies in the convex hull of the means \citep[Corollary~1]{RayLindsay2005}, and in fact the entire mode-counting problem reduces to the affine span of the means.

\begin{theorem}[Affine-rank reduction]\label{thm:affine-rank}
Assume that $\Sigma_1=\cdots=\Sigma_k=\Sigma$ and let
\[
r=\dim \aff\{\mu_1,\dots,\mu_k\}.
\]
Then, after an invertible affine change of variables and an orthogonal decomposition $\R^d=\R^r\times\R^{d-r}$, the density factors in the form
\[
\widetilde \Phi(u,v)=C e^{-\|v\|^2/2}G(u),
\]
where $C>0$ is constant and $G$ is an $r$-dimensional homoscedastic Gaussian mixture with the same number of components. When $r=0$, $G$ is understood as the unit density on $\R^0$.
\end{theorem}

\begin{proof}
See Appendix~\ref{subsec:affine-rank-proof}.
\end{proof}

\begin{corollary}\label{cor:affine-rank-consequence}
Under the hypotheses of Theorem~\ref{thm:affine-rank}, all critical points and all modes of $\Phi$ lie in the affine span of the means, and the numbers of critical points, points of $\Crit_{\mathrm{nd}}(\Phi)$, and modes of $\Phi$ agree exactly with the corresponding numbers for $G$.
\end{corollary}

\begin{proof}
For
\[
\widetilde \Phi(u,v)=C e^{-\|v\|^2/2}G(u),
\]
one has
\[
\nabla_u\widetilde \Phi(u,v)=Ce^{-\|v\|^2/2}\nabla G(u),
\qquad
\nabla_v\widetilde \Phi(u,v)=-Ce^{-\|v\|^2/2}G(u)v.
\]
Thus $\nabla\widetilde \Phi(u,v)=0$ if and only if $v=0$ and $\nabla G(u)=0$. At such a point the Hessian splits block diagonally, and its lower-right block is $-CG(u)I_{d-r}$, so nondegeneracy and local maximality are preserved as well.
\end{proof}

Combining Corollary~\ref{cor:affine-rank-consequence} with Corollary~\ref{cor:hom-special} gives the direct $d$-dimensional homoscedastic bound in the intrinsic dimension $r$.

\begin{corollary}[Affine-rank homoscedastic bound]\label{cor:affine-rank-bound}
Under the hypotheses of Theorem~\ref{thm:affine-rank}, if $r=0$, then $\Phi$ has exactly one nondegenerate critical point and one mode. If $r\ge1$, then
\[
|\Mode_{\mathrm{nd}}(\Phi)|\le |\Crit_{\mathrm{nd}}(\Phi)|\le U_{\mathrm{hom}}(r,k)=2^{\,r+\binom{k-1}{2}}\bigl(r+\min(r,k-1)+1\bigr)^{k-1}.
\]
If $r\ge1$ and the modal set is finite, then
\[
|\Mode(\Phi)|\le \left\lfloor\frac{U_{\mathrm{hom}}(r,k)+1}{2}\right\rfloor.
\]
\end{corollary}

\begin{proof}
If $r=0$, then all means coincide and the homoscedastic mixture is a single Gaussian density, so it has exactly one nondegenerate critical point and one mode. Assume $r\ge1$. Apply Corollary~\ref{cor:hom-special} to the $r$-dimensional mixture $G$ from Theorem~\ref{thm:affine-rank}. The count identities follow from Corollary~\ref{cor:affine-rank-consequence}.
\end{proof}

In the homoscedastic setting the augmented reduced ratio system itself simplifies further, because $M(y)$ becomes a scalar multiple of the common precision matrix.

\begin{theorem}[Augmented homoscedastic reduced bound]\label{thm:hom-augmented-bound}
Let $\Phi$ be a homoscedastic Gaussian mixture with $k\ge 2$ components. Then
\[
|\Mode_{\mathrm{nd}}(\Phi)|\le |\Crit_{\mathrm{nd}}(\Phi)|\le
U_{\mathrm{aug,hom}}(k)=2^{\binom{k-1}{2}+1}(2k)^{k-1}.
\]
\end{theorem}

\begin{proof}
For the dimension-free bound, it suffices to work after an invertible affine change of variables with common covariance $I$, because such a change preserves critical points, nondegeneracy, and modes. In the notation of the reduced system,
\[
M(y)=\left(1+\sum_{j=1}^m y_j\right)I,
\qquad
X(y)=\frac{\mu_k+\sum_{j=1}^m y_j\mu_j}{1+\sum_{j=1}^m y_j}.
\]
Moreover $q_i(x)$ is affine in $x$. Hence there are constants $a_i\in\R$ and affine linear functions $L_i(y)$ such that
\[
q_i(X(y))=a_i+\frac{L_i(y)}{1+\sum_{j=1}^m y_j}.
\]
Introduce $z$ and define
\[
\eta_i(y,z)=\exp(a_i+zL_i(y)).
\]
The augmented homoscedastic system is
\begin{equation}\label{eq:aug-hom-system}
z\left(1+\sum_{j=1}^m y_j\right)-1=0,
\qquad
y_i-\beta_i\eta_i(y,z)=0,\quad i=1,\dots,m.
\end{equation}
The same bijection and nondegeneracy proof as in Lemma~\ref{lem:aug-bijection} applies, with $D(y)=1+\sum_jy_j$.

The functions $\eta_1,\dots,\eta_m$ form a Pfaff chain of length $m=k-1$ and degree at most $2$: the exponent $a_i+zL_i(y)$ has total degree at most $2$, its derivatives have degree at most $1$, and
\[
\frac{\partial\eta_i}{\partial w}=\eta_i\frac{\partial}{\partial w}\{a_i+zL_i(y)\}
\]
is therefore polynomial in $(y,z,\eta_i)$ of degree at most $2$. In \eqref{eq:aug-hom-system}, the first equation has degree $2$ and the remaining equations have degree $1$. Applying Theorem~\ref{thm:Khov} with
\[
n=k,\qquad q=k-1,\qquad M=2,\qquad p_0=2,\qquad p_1=\cdots=p_m=1,
\]
gives
\[
L=(2-1)+(k-1)\cdot2+1=2k
\]
and hence
\[
2^{\binom{k-1}{2}}\cdot 2\cdot(2k)^{k-1}
=
2^{\binom{k-1}{2}+1}(2k)^{k-1}.
\]
This proves the dimension-free augmented homoscedastic bound.
\end{proof}

\begin{corollary}[Pointwise best homoscedastic bounds]\label{cor:best-hom-bound}
Let $\Phi$ be a homoscedastic Gaussian mixture with $k\ge2$ components, and let
\[
r=\dim\aff\{\mu_1,\dots,\mu_k\}.
\]
If $r=0$, then $\Phi$ has exactly one nondegenerate critical point and one mode. If $r\ge1$, then
\[
|\Crit_{\mathrm{nd}}(\Phi)|
\le
U_{\mathrm{best,hom}}(r,k)=
\min\{U_{\mathrm{hom}}(r,k),U_{\mathrm{aug}}(r,k),U_{\mathrm{aug,hom}}(k)\}.
\]
If $r\ge1$ and the modal set is finite, then
\[
|\Mode(\Phi)|\le
\left\lfloor\frac{U_{\mathrm{best,hom}}(r,k)+1}{2}\right\rfloor.
\]
\end{corollary}

\begin{proof}
If $r=0$, then all means coincide, so the mixture is a single Gaussian density and has one nondegenerate critical point and one mode. If $r\ge1$, apply Theorem~\ref{thm:affine-rank} and Corollary~\ref{cor:affine-rank-consequence}. The reduced $r$-dimensional mixture $G$ satisfies the direct homoscedastic bound $U_{\mathrm{hom}}(r,k)$, the general augmented bound $U_{\mathrm{aug}}(r,k)$, and the dimension-free augmented homoscedastic bound $U_{\mathrm{aug,hom}}(k)$. Taking the minimum gives the stated nondegenerate-critical bound. The finite-mode statement follows from Proposition~\ref{prop:finite-mode-transfer} and Lemma~\ref{lem:generic-tilt}.
\end{proof}

\begin{remark}\label{rem:comparison-aeh}
The published bound of \citet{AmendolaEngstromHaase2019} is
\[
U_{\mathrm{AEH}}(d,k)=2^{\,d+\binom{k}{2}}(5+3d)^k,
\]
whereas our direct $d$-dimensional bound is
\[
U_{\mathrm{het}}(d,k)=2^{\,d+\binom{k-1}{2}}\bigl(d+2\min(d,k-1)+1\bigr)^{k-1}.
\]
The normalization by one reference component replaces the $k$ exponentials appearing in the direct critical-point system by the $k-1$ ratios $\rho_1,\dots,\rho_{k-1}$. This lowers the Pfaff-chain length by one, and correspondingly replaces the prefactor $2^{\binom{k}{2}}$ by $2^{\binom{k-1}{2}}$ and the outer exponent $k$ by $k-1$. Moreover,
\[
d+2\min(d,k-1)+1\le 3d+1<5+3d,
\]
so
\[
U_{\mathrm{het}}(d,k)<2^{-(k-1)}(5+3d)^{-1}U_{\mathrm{AEH}}(d,k)<U_{\mathrm{AEH}}(d,k)
\]
for every $d\ge 1$ and $k\ge 2$. The augmented bound $U_{\mathrm{aug}}$ gives a second regime of improvement: for each fixed $k$, it is polynomial in $d$ and hence eventually beats both $U_{\mathrm{het}}$ and $U_{\mathrm{AEH}}$.
\end{remark}

\subsection{A numerical comparison of the explicit upper bounds}\label{subsec:upper-numerical-comparison}

We now compare the competing upper-bound families numerically. The augmented bound $U_{\mathrm{aug}}(d,k)$ is not uniformly smaller than $U_{\mathrm{het}}(d,k)$ for every pair $(d,k)$, but it eventually becomes smaller for each fixed $k$. Define
\[
d_\ast(k)=\min\left\{d\in\mathbb N: U_{\mathrm{aug}}(d,k)\le U_{\mathrm{het}}(d,k)\right\}.
\]
Table~\ref{tab:crossover} records these first crossover dimensions for $2\le k\le 11$. Thus, for example, when $k=5$ the direct $d$-dimensional bound $U_{\mathrm{het}}(d,5)$ is smaller through to $d=14$, whereas the augmented bound $U_{\mathrm{aug}}(d,5)$ first becomes smaller at $d=15$.

The same phenomenon occurs when one compares $U_{\mathrm{aug}}$ with the published bound $U_{\mathrm{AEH}}$. Define
\[
d_{\mathrm{AEH}}(k)=\min\left\{d\in\mathbb N: U_{\mathrm{aug}}(d,k)\le U_{\mathrm{AEH}}(d,k)\right\}.
\]
By definition, $U_{\mathrm{aug}}(d,k)>U_{\mathrm{AEH}}(d,k)$ for $1\le d<d_{\mathrm{AEH}}(k)$, and Table~\ref{tab:crossover} records the first dimension at which $U_{\mathrm{aug}}(d,k)$ is no larger than $U_{\mathrm{AEH}}(d,k)$.

\begin{table}[htbp]
\centering
\caption{Crossover dimensions for the augmented reduced bound.}
\label{tab:crossover}
\small
\begin{tabular}{lcccccccccc}
\toprule
$k$ & 2 & 3 & 4 & 5 & 6 & 7 & 8 & 9 & 10 & 11 \\
\midrule
$d_\ast(k)$ & 3 & 7 & 10 & 15 & 19 & 24 & 29 & 34 & 39 & 45 \\
$d_{\mathrm{AEH}}(k)$ & 1 & 1 & 1 & 1 & 1 & 4 & 7 & 10 & 13 & 16 \\
\bottomrule
\end{tabular}
\end{table}

For the table below, write
\[
\begin{aligned}
U_{\mathrm{best}}(d,k)
&=\min\left\{U_{\mathrm{het}}(d,k),U_{\mathrm{aug}}(d,k)\right\},\qquad\text{and}\\
U_{\mathrm{best,hom}}(r,k)
&=\min\left\{U_{\mathrm{hom}}(r,k),U_{\mathrm{aug}}(r,k),U_{\mathrm{aug,hom}}(k)\right\}\qquad (r\ge1).
\end{aligned}
\]
Table~\ref{tab:upper-comparison} compares the conjectural AIM value with these pointwise best explicit nondegenerate-critical-point upper bounds and with the individual bound families. When the modal set is finite, the corresponding mode bounds are obtained by replacing $U_{\mathrm{best}}$ and $U_{\mathrm{best,hom}}$ by $\lfloor(U_{\mathrm{best}}+1)/2\rfloor$ and $\lfloor(U_{\mathrm{best,hom}}+1)/2\rfloor$, respectively. The homoscedastic block is a formal comparison of the intrinsic-dimension formulas; only rows satisfying $r\le k-1$ are realizable as affine ranks of $k$ component means.

\begin{table}[htbp]
\centering
\caption{Representative values of the conjectural AIM count and the explicit upper bounds in this manuscript.}
\label{tab:upper-comparison}
\scriptsize
\setlength{\tabcolsep}{4pt}
\begin{tabular}{ccccccc}
\toprule
\multicolumn{7}{c}{General heteroscedastic setting} \\
\midrule
$d$ & $k$ & $m_{\mathrm{AIM}}(d,k)$ & $U_{\mathrm{best}}(d,k)$ & $U_{\mathrm{het}}(d,k)$ & $U_{\mathrm{aug}}(d,k)$ & $U_{\mathrm{AEH}}(d,k)$ \\
\midrule
2 & 2 & $3$ & $20$ & $20$ & $27$ & $968$ \\
2 & 3 & $6$ & $392$ & $392$ & $1.35\times 10^{3}$ & $4.26\times 10^{4}$ \\
2 & 4 & $10$ & $1.10\times 10^{4}$ & $1.10\times 10^{4}$ & $2.22\times 10^{5}$ & $3.75\times 10^{6}$ \\
3 & 3 & $10$ & $1.02\times 10^{3}$ & $1.02\times 10^{3}$ & $3.20\times 10^{3}$ & $1.76\times 10^{5}$ \\
3 & 4 & $20$ & $6.40\times 10^{4}$ & $6.40\times 10^{4}$ & $7.02\times 10^{5}$ & $1.97\times 10^{7}$ \\
4 & 4 & $35$ & $1.70\times 10^{5}$ & $1.70\times 10^{5}$ & $1.72\times 10^{6}$ & $8.55\times 10^{7}$ \\
4 & 5 & $70$ & $2.92\times 10^{7}$ & $2.92\times 10^{7}$ & $1.31\times 10^{9}$ & $2.33\times 10^{10}$ \\
5 & 5 & $126$ & $7.87\times 10^{7}$ & $7.87\times 10^{7}$ & $3.27\times 10^{9}$ & $1.05\times 10^{11}$ \\
10 & 4 & $286$ & $4.02\times 10^{7}$ & $4.02\times 10^{7}$ & $4.02\times 10^{7}$ & $9.83\times 10^{10}$ \\
10 & 5 & $1001$ & $8.54\times 10^{9}$ & $8.54\times 10^{9}$ & $6.76\times 10^{10}$ & $5.51\times 10^{13}$ \\
\midrule
\multicolumn{7}{c}{Homoscedastic intrinsic-dimension formulas} \\
\midrule
$r$ & $k$ & $m_{\mathrm{AIM}}(r,k)$ & $U_{\mathrm{best,hom}}(r,k)$ & $U_{\mathrm{hom}}(r,k)$ & $U_{\mathrm{aug}}(r,k)$ & $U_{\mathrm{aug,hom}}(k)$ \\
\midrule
1 & 2 & $2$ & $6$ & $6$ & $12$ & $8$ \\
1 & 3 & $3$ & $36$ & $36$ & $400$ & $144$ \\
2 & 3 & $6$ & $144$ & $200$ & $1.35\times 10^{3}$ & $144$ \\
1 & 4 & $4$ & $432$ & $432$ & $4.39\times 10^{4}$ & $8.19\times 10^{3}$ \\
2 & 4 & $10$ & $4.00\times 10^{3}$ & $4.00\times 10^{3}$ & $2.22\times 10^{5}$ & $8.19\times 10^{3}$ \\
3 & 4 & $20$ & $8.19\times 10^{3}$ & $2.20\times 10^{4}$ & $7.02\times 10^{5}$ & $8.19\times 10^{3}$ \\
1 & 5 & $5$ & $1.04\times 10^{4}$ & $1.04\times 10^{4}$ & $1.34\times 10^{7}$ & $1.28\times 10^{6}$ \\
2 & 5 & $15$ & $1.60\times 10^{5}$ & $1.60\times 10^{5}$ & $1.02\times 10^{8}$ & $1.28\times 10^{6}$ \\
3 & 5 & $35$ & $1.23\times 10^{6}$ & $1.23\times 10^{6}$ & $4.30\times 10^{8}$ & $1.28\times 10^{6}$ \\
4 & 5 & $70$ & $1.28\times 10^{6}$ & $6.72\times 10^{6}$ & $1.31\times 10^{9}$ & $1.28\times 10^{6}$ \\
\bottomrule
\end{tabular}
\end{table}

\section{Lower bounds on the number of modes}\label{sec:lower-mode-bounds}
To avoid confusion with the modal-set notation $\Mode(\Phi)$, we continue to use the finite-modal-set counting notation $m(d,k)$ from Remark~\ref{rem:m-vs-Mode}. When a homoscedastic restriction is imposed, we write $m_{\mathrm{hom}}(d,k)$ for the corresponding supremum. The upper bounds above should be compared with the lower bound of \citet[Theorem~4.1]{AmendolaEngstromHaase2019}:
\[
L_{\mathrm{AEH}}(d,k)=\binom{k}{d}+k,
\]
and our subsequent improvements. The next proposition shows that lower-dimensional constructions may be lifted to higher dimension without losing modes.

\begin{proposition}[Monotonicity under dimension lifting]\label{prop:ambient-monotonicity}
For integers $1\le r\le d$ and $k\ge 1$,
\[
m(d,k)\ge m(r,k).
\]
The same statement holds in the homoscedastic class. Moreover, the dimension-lifting construction preserves finite modal sets and isolated modes.
\end{proposition}

\begin{proof}
See Appendix~\ref{subsec:lower-bound-proofs}.
\end{proof}

\begin{corollary}[Lifted lower-dimensional bound]\label{cor:lifted-lower}
For $d\ge 2$ and $k\ge 2$,
\[
m(d,k)\ge L_{\mathrm{bin}}(d,k)=k+\max_{2\le r\le \min(d,k)}\binom{k}{r}.
\]
In particular, for $k\ge 4$ and $d\ge \lfloor k/2\rfloor$,
\[
m(d,k)\ge k+\binom{k}{\lfloor k/2\rfloor}.
\]
\end{corollary}

\begin{proof}
See Appendix~\ref{subsec:lower-bound-proofs}.
\end{proof}

Since the maximum in $L_{\mathrm{bin}}$ ranges over all $2\le r\le \min(d,k)$, the inequality
\[
L_{\mathrm{bin}}(d,k)\ge L_{\mathrm{AEH}}(d,k)
\]
is immediate, and the inequality is strict exactly when the maximum of $\binom{k}{r}$ over $2\le r\le \min(d,k)$ is larger than $\binom{k}{d}$, with $\binom{k}{d}=0$ if $d>k$.

A second lower-bound family comes from two operations. First, one may add a far-away Gaussian component and force at least one new mode, and second, one may take Cartesian products of Gaussian mixtures.

\begin{proposition}[Remote padding and Cartesian products]\label{prop:padding-products}
The following statements hold.
\begin{enumerate}[(a)]
    \item If a $k$-component Gaussian mixture in $\R^d$ has at least $M\ge1$ isolated modes, then for every integer $\ell\ge 0$ there exists a $(k+\ell)$-component Gaussian mixture in $\R^d$ with finite modal set and at least $M+\ell$ nondegenerate modes.
    \item If Gaussian mixtures with parameter pairs $(d_1,k_1)$ and $(d_2,k_2)$ have at least $M_1\ge1$ and $M_2\ge1$ isolated modes, respectively, then there exists a Gaussian mixture with parameter pair $(d_1+d_2,k_1k_2)$, finite modal set, and at least $M_1M_2$ nondegenerate modes.
\end{enumerate}
The same statements hold in the homoscedastic class.
\end{proposition}

\begin{proof}
See Appendix~\ref{subsec:lower-bound-proofs}.
\end{proof}

\begin{theorem}[Padding-product lower bound]\label{thm:padding-product-lower}
Let $d\ge 1$ and $k\ge 2$. For every integer $s$ with
\[
1\le s\le \min\bigl(d,\lfloor \log_2 k\rfloor\bigr),
\]
let $\lfloor d/s\rfloor$ and $d\bmod s$ be the quotient and remainder on division of $d$ by $s$. Then
\[
m(d,k)\ge k-2^s+\left(\left\lfloor\frac{d}{s}\right\rfloor+1\right)^{s-(d\bmod s)}\left(\left\lfloor\frac{d}{s}\right\rfloor+2\right)^{d\bmod s}.
\]
\end{theorem}

\begin{proof}
See Appendix~\ref{subsec:lower-bound-proofs}.
\end{proof}

\begin{corollary}[Padding-product family]\label{cor:padding-product-family}
For $d\ge1$ and $k\ge2$,
\[
m(d,k)\ge L_{\mathrm{pp}}(d,k)=
\max_{1\le s\le \min(d,\lfloor \log_2 k\rfloor)}
\left(k-2^s+\left(\left\lfloor\frac{d}{s}\right\rfloor+1\right)^{s-(d\bmod s)}\left(\left\lfloor\frac{d}{s}\right\rfloor+2\right)^{d\bmod s}\right).
\]
\end{corollary}

\begin{proof}
Taking the maximum over the admissible integers $s$ in Theorem~\ref{thm:padding-product-lower} gives the definition of $L_{\mathrm{pp}}(d,k)$ and the stated inequality.
\end{proof}

\begin{corollary}[Linear padding-product lower bound]\label{cor:linear-padding-lower}
For $d\ge1$ and $k\ge2$,
\[
L_{\mathrm{pp}}(d,k)\ge d+k-1.
\]
In particular, $m(d,k)\ge d+k-1$.
\end{corollary}

\begin{proof}
In the definition of $L_{\mathrm{pp}}(d,k)$, choose $s=1$. Then $\lfloor d/1\rfloor=d$ and $d\bmod 1=0$, so the corresponding value is $k-2+(d+1)=d+k-1$. The lower bound for $m(d,k)$ follows from Corollary~\ref{cor:padding-product-family}.
\end{proof}

\begin{corollary}\label{cor:best-lower}
For all $d\ge 2$ and $k\ge 2$,
\[
m(d,k)\ge L_{\mathrm{best}}(d,k)=\max\bigl\{L_{\mathrm{bin}}(d,k),L_{\mathrm{pp}}(d,k)\bigr\}.
\]
\end{corollary}

\begin{proof}
This is immediate from Corollaries~\ref{cor:lifted-lower} and \ref{cor:padding-product-family}.
\end{proof}

\begin{corollary}[Eventual dominance of the padding-product family]\label{cor:pp-eventual-dominance}
For every fixed $k\ge 4$, the family $L_{\mathrm{pp}}(d,k)$ eventually dominates $L_{\mathrm{bin}}(d,k)$ as $d\to\infty$.
\end{corollary}

\begin{proof}
Let $s_0=\lfloor \log_2 k\rfloor$. Then $s_0\ge 2$, and the definition of $L_{\mathrm{pp}}$ gives
\[
L_{\mathrm{pp}}(d,k)\ge k-2^{s_0}+\left(\frac{d}{s_0}\right)^{s_0}
\]
for all $d\ge s_0$. On the other hand, if $d\ge \lfloor k/2\rfloor$, then
\[
L_{\mathrm{bin}}(d,k)=k+\binom{k}{\lfloor k/2\rfloor},
\]
which is constant in $d$. Since $s_0\ge 2$, the displayed lower bound for $L_{\mathrm{pp}}(d,k)$ grows at least quadratically in $d$, and therefore eventually exceeds $L_{\mathrm{bin}}(d,k)$.
\end{proof}

\subsection{Seed closures}\label{subsec:seed-closures}

The preceding padding-product construction is a special case of a general closure principle. This formulation is useful because any new seed construction can be inserted without changing the proof.

\begin{definition}[Seed triple]
A seed triple is a triple of positive integers $(d_0,k_0,M_0)$ for which there exists a $k_0$-component Gaussian mixture in $\R^{d_0}$ with finite modal set and at least $M_0$ modes. A homoscedastic seed triple is defined in the same way with the additional requirement that the seed mixture is homoscedastic.
\end{definition}

\begin{theorem}[Product closure of seed constructions]\label{thm:seed-closure}
Let $\mathcal S$ be any collection of seed triples. Define
\[
L_{\mathcal S}(d,k)=
\max\left\{
 k-K+\prod_{j=1}^t M_j:
 t\ge0,\ 
 (d_j,k_j,M_j)\in\mathcal S,\ 
 \sum_{j=1}^t d_j\le d,\ 
 K=\prod_{j=1}^t k_j\le k
\right\},
\]
with empty products equal to $1$, empty sums equal to $0$, and inadmissible choices ignored. Then
\[
m(d,k)\ge L_{\mathcal S}(d,k).
\]
The same statement holds for $m_{\mathrm{hom}}(d,k)$ when $\mathcal S$ is a collection of homoscedastic seed triples.
\end{theorem}

\begin{proof}
Choose an admissible finite list of seed triples $(d_j,k_j,M_j)$. If the list is empty, start from any one-component Gaussian in $\R^d$ and use Proposition~\ref{prop:padding-products}(a) with $\ell=k-1$, obtaining a $k$-component mixture with finite modal set and at least $k$ modes; this is exactly the empty-list value $k-1+1$. Thus assume the list is nonempty. For each seed triple choose a witnessing mixture with finite modal set and at least $M_j$ modes; these modes are isolated. Repeatedly applying Proposition~\ref{prop:padding-products}(b) gives a mixture with parameter pair
\[
\left(\sum_{j=1}^t d_j,\prod_{j=1}^t k_j\right)
\]
with finite modal set and at least $\prod_{j=1}^t M_j$ modes.
If $\sum_jd_j<d$, apply Proposition~\ref{prop:ambient-monotonicity}, whose construction preserves finite modal sets and isolated modes. If $K=\prod_jk_j<k$, apply Proposition~\ref{prop:padding-products}(a) with $\ell=k-K$. The resulting $k$-component mixture in $\R^d$ has finite modal set and at least
\[
\prod_{j=1}^tM_j+(k-K)
\]
modes. Taking the maximum over admissible lists proves the theorem. The homoscedastic proof is identical, using the homoscedastic parts of Propositions~\ref{prop:ambient-monotonicity} and \ref{prop:padding-products}.
\end{proof}

The two-component construction of \citet{RayRen2012} gives seed triples $(s,2,s+1)$ for every $s\ge1$, and Theorem~\ref{thm:padding-product-lower} is recovered by applying Theorem~\ref{thm:seed-closure} to this family. The next theorem gives a simple homoscedastic seed family, closely related to the regular-simplex constructions discussed in \citet{CarreiraPerpinanWilliams2003}.

\begin{theorem}\label{thm:simplex-seed}
For every integer $K\ge3$, $(K-1,K,K+1)$ is a homoscedastic seed triple.
\end{theorem}

\begin{proof}
See Appendix~\ref{subsec:simplex-seed-proof}.
\end{proof}

\begin{corollary}\label{cor:hom-simplex-product}
In the homoscedastic setting,
\[
m_{\mathrm{hom}}(d,k)\ge
\max_{\substack{s\ge0,\ K_1,\dots,K_s\ge3\\
\sum_{j=1}^sK_j-s\le d,\ \prod_{j=1}^sK_j\le k}}
\left(k-\prod_{j=1}^sK_j+\prod_{j=1}^s(K_j+1)\right),
\]
where empty products are interpreted as $1$. In particular,
\[
m_{\mathrm{hom}}(d,k)\ge k+1
\qquad(k\ge3,\ d\ge k-1).
\]
\end{corollary}

\begin{proof}
Apply Theorem~\ref{thm:seed-closure} to the homoscedastic seed family $(K-1,K,K+1)$ from Theorem~\ref{thm:simplex-seed}. The displayed maximum is exactly the resulting seed-closure lower bound. For the final inequality, choose the single seed with $K_1=k$; it is admissible when $d\ge k-1$, and its value is $k-k+(k+1)=k+1$.
\end{proof}

\subsection{A numerical comparison of the lower-bound families}\label{subsec:lower-numerical-comparison}

We now compare the two simple closed-form lower-bound families numerically. Define
\[
d_{\mathrm{pp}}(k)=\min\left\{d\ge 2: L_{\mathrm{pp}}(d,k)>L_{\mathrm{bin}}(d,k)\right\}.
\]
Table~\ref{tab:lower-crossover} records the first crossover dimensions for $2\le k\le 11$, and Table~\ref{tab:lower-comparison} compares the conjectural AIM value with the published lower bound $L_{\mathrm{AEH}}$, the lifted lower-dimensional family $L_{\mathrm{bin}}$, the padding-product family $L_{\mathrm{pp}}$, and the previously defined pointwise maximum $L_{\mathrm{best}}(d,k)$.

\begin{table}[htbp]
\centering
\caption{Crossover dimensions for the lower-bound families.}
\label{tab:lower-crossover}
\small
\begin{tabular}{lcccccccccc}
\toprule
$k$ & 2 & 3 & 4 & 5 & 6 & 7 & 8 & 9 & 10 & 11 \\
\midrule
$d_{\mathrm{pp}}(k)$ & 3 & 5 & 5 & 6 & 8 & 11 & 10 & 13 & 17 & 21 \\
\bottomrule
\end{tabular}
\end{table}

\begin{table}[htbp]
\centering
\caption{Representative values of the conjectural AIM count and the lower-bound families.}
\label{tab:lower-comparison}
\small
\begin{tabular}{ccccccc}
\toprule
$d$ & $k$ & $m_{\mathrm{AIM}}(d,k)$ & $L_{\mathrm{AEH}}(d,k)$ & $L_{\mathrm{bin}}(d,k)$ & $L_{\mathrm{pp}}(d,k)$ & $L_{\mathrm{best}}(d,k)$ \\
\midrule
2 & 2 & $3$ & $3$ & $3$ & $3$ & $3$ \\
2 & 3 & $6$ & $6$ & $6$ & $4$ & $6$ \\
2 & 4 & $10$ & $10$ & $10$ & $5$ & $10$ \\
3 & 3 & $10$ & $4$ & $6$ & $5$ & $6$ \\
3 & 4 & $20$ & $8$ & $10$ & $6$ & $10$ \\
4 & 4 & $35$ & $5$ & $10$ & $9$ & $10$ \\
4 & 5 & $70$ & $10$ & $15$ & $10$ & $15$ \\
5 & 5 & $126$ & $6$ & $15$ & $13$ & $15$ \\
10 & 4 & $286$ & $4$ & $10$ & $36$ & $36$ \\
10 & 5 & $1001$ & $5$ & $15$ & $37$ & $37$ \\
\bottomrule
\end{tabular}
\end{table}
\section{Bounding the number of connected components}\label{sec:modal-regions}

\subsection{Connected components of the critical set}\label{subsec:critical-components}

Since $\Mode(\Phi)\subseteq \Crit(\Phi)$, any bound on the number of connected components of $\Crit(\Phi)$ localizes the possible connected modal regions, but it does not by itself bound their number. The following theorem gives the corresponding bound for the connected components of the critical set. 

\begin{theorem}[Critical-set bound]\label{thm:crit-components}
For every Gaussian mixture with $k\ge 2$ components in $\R^d$, the set $\Crit(\Phi)$ has at most
\[
U_{\mathrm{crit}}(d,k)=2^{\binom{k-1}{2}+2}5^{d-1}(6d-2)^{k-1}
\]
connected components.
\end{theorem}

\begin{proof}
Write the normalized critical equations as
\[
F_r(x)=Q_r(x,\rho(x))=0,\qquad r=1,\dots,d.
\]
These are Pfaffian equations on the global Pfaffian domain $\R^d$ with common chain length $k-1$, chain degree $\alpha=2$, and Pfaffian degree $\beta=2$. Theorem~\ref{thm:GVcc} therefore gives
\[
2^{\binom{k-1}{2}+1}\cdot 2\cdot (2+4-1)^{d-1}\cdot \bigl((2d-1)(2+2)-2d+2\bigr)^{k-1},
\]
which simplifies to the stated formula.
\end{proof}

\begin{remark}\label{rem:no-example}
We are not aware of an example of a finite Gaussian mixture whose modal set contains a connected region with more than one point. Equally, we do not know an impossibility theorem ruling such regions out.
\end{remark}

\subsection{Homoscedastic refinement}\label{subsec:crit-hom}

The same Pfaffian argument yields an improved critical-set component bound in the homoscedastic case.

\begin{corollary}[Homoscedastic critical-set bound]\label{cor:crit-hom}
Assume that $\Sigma_1=\cdots=\Sigma_k=\Sigma$, and let
\[
r=\dim \aff\{\mu_1,\dots,\mu_k\}.
\]
If $r=0$, then $\Crit(\Phi)$ has exactly one connected component. If $r\ge 1$, then $\Crit(\Phi)$ has at most
\[
U_{\mathrm{crit,hom}}(r,k)=2^{\binom{k-1}{2}+2}4^{r-1}(4r-1)^{k-1}
\]
connected components.
\end{corollary}

\begin{proof}
If $r=0$, then all means coincide, say $\mu_1=\cdots=\mu_k=\mu$, and the homoscedastic mixture reduces to the single Gaussian density
\[
\Phi(x)=\phi(x;\mu,\Sigma).
\]
Hence $\Crit(\Phi)=\{\mu\}$.

Assume now that $r\ge 1$. Apply Theorem~\ref{thm:affine-rank}; as in the proof of Corollary~\ref{cor:affine-rank-consequence}, the critical set is identified with $\Crit(G)\times\{0\}$ after the affine change of variables, so the number of connected components is unchanged by the reduction to the $r$-dimensional homoscedastic mixture $G$. In that reduced system the equations are Pfaffian on the global Pfaffian domain $\R^r$, the Pfaff chain degree is $\alpha=1$, and the defining Pfaffian degree remains $\beta=2$. Theorem~\ref{thm:GVcc} with $(\alpha,\beta)=(1,2)$ gives
\[
2^{\binom{k-1}{2}+1}\cdot 2\cdot 4^{r-1}\cdot (4r-1)^{k-1},
\]
which is the stated bound.
\end{proof}

\section{Conclusion}\label{sec:discussion}

The main quantitative outcome of the manuscript is a pointwise improvement over the published upper bound of \citet{AmendolaEngstromHaase2019} together with an augmented reduced bound that becomes sharper in high dimension. The direct Pfaffian estimate $U_{\mathrm{het}}(d,k)$ is smaller than $U_{\mathrm{AEH}}(d,k)$ for every $d\ge 1$ and $k\ge 2$, while the augmented estimate $U_{\mathrm{aug}}(d,k)$ is eventually smaller than both $U_{\mathrm{het}}(d,k)$ and $U_{\mathrm{AEH}}(d,k)$ for each fixed $k$ as $d$ grows. The resulting nondegenerate-critical-point bound is the pointwise minimum $\min\{U_{\mathrm{het}}(d,k),U_{\mathrm{aug}}(d,k)\}$. For finite modal sets, exponential tilting and a Morse-theoretic argument improve the corresponding mode bound by a factor of two.

In the homoscedastic setting, $U_{\mathrm{hom}}(d,k)$ improves the direct bound further, the affine-rank reduction replaces the full Euclidean dimension $d$ by the intrinsic affine rank $r$ of the means, and the augmented homoscedastic bound $U_{\mathrm{aug,hom}}(k)$ supplies a dimension-free regime of improvement. Thus, when $r\ge1$, the intrinsic homoscedastic nondegenerate-critical-point bound is
\[
\min\{U_{\mathrm{hom}}(r,k),U_{\mathrm{aug}}(r,k),U_{\mathrm{aug,hom}}(k)\},
\]
with the same factor of two reduction for finite modal sets; when $r=0$, there is exactly one critical point and one mode.

On the lower-bound side, the lifted family $L_{\mathrm{bin}}(d,k)$ improves the published lower bound in regimes where the maximum of $\binom{k}{r}$ over $2\le r\le \min(d,k)$ is larger than $\binom{k}{d}$, while the padding-product family $L_{\mathrm{pp}}(d,k)$ eventually dominates $L_{\mathrm{bin}}(d,k)$ for every fixed $k\ge 4$. The seed-closure theorem packages the same dimension-lifting, product, and padding operations in a reusable form. In particular, any future seed construction immediately yields a new lower bound by insertion into $L_{\mathcal S}(d,k)$, and the regular-simplex seed gives an example of such a homoscedastic lower-bound family.

For connected sets, we obtain explicit upper bounds $U_{\mathrm{crit}}(d,k)$ and $U_{\mathrm{crit,hom}}(r,k)$ for the number of connected components of the critical set. These bounds do not by themselves control connected modal regions, because a disconnected subset of a connected critical component may have several connected components. At present we are not aware of either an example of a finite Gaussian mixture with a nontrivial connected modal region or an impossibility theorem ruling such regions out.

We presently do not resolve the conjecture discussed at the 2011 American Institute of Mathematics workshop on Singular Learning Theory, that
\[
m_{\mathrm{AIM}}(d,k)=\binom{d+k-1}{d}
\]
should be the exact maximum possible number of modes of a $d$-dimensional Gaussian mixture with $k$ components. The current gap remains substantial. The ridgeline analysis of \citet{RayLindsay2005}, the sharp two-component theorem of \citet{RayRen2012}, and the parameter-sensitive bivariate refinement of \citet{KabataMatsumotoUchidaUeki2025} show that exact answers depend delicately on the covariance-dependent geometry even in low-dimensional and low-complexity regimes. Outside those settings, however, we still lack general counting theorems that exploit enough of the problem-specific geometry to force the conjectural formula.

\appendix
\renewcommand{\theHsection}{appendix.\Alph{section}}
\renewcommand{\theHsubsection}{appendix.\Alph{section}.\arabic{subsection}}
\renewcommand{\theHequation}{appendix.\Alph{section}.\arabic{equation}}

\section{External results}\label{app:external}

This appendix records the external bounds from fewnomial and semi-Pfaffian geometry on which the manuscript depends, together with the standard Morse handle theorem used in Section~\ref{subsec:morse-factor}. Theorem~\ref{thm:Khov} records the form of \citet[Section 3.12, Corollary~5]{Khovanskii1991} used here, Theorem~\ref{thm:GVcc} records the global-domain form of \citet[Corollary~3.3]{GabrielovVorobjov2004} used here, and Theorem~\ref{thm:morse-handle} is the Morse handle-attachment theorem in the form used below. The handle-decomposition background follows \citet[Chapter~3]{Knudson2015}, and the simultaneous form for several critical points on one critical level is recorded in \citet[Remark~2.9]{Nicolaescu2011}.

\begin{theorem}[Pfaff-chain root bound]\label{thm:Khov}
Let $U\subseteq\R^n$ be a connected open set, and let $y_1,\dots,y_q$ be a Pfaff chain of length $q$ and degree $M$ on $U$, meaning that for every $1\le i\le q$ and $1\le j\le n$,
\[
\frac{\partial y_i}{\partial x_j}(x)=P_{ij}(x,y_1(x),\dots,y_i(x))
\]
with $P_{ij}$ a polynomial of total degree at most $M$. Let $Q_1,\dots,Q_n$ be polynomials of degrees $p_1,\dots,p_n$ in the variables $x_1,\dots,x_n,y_1,\dots,y_q$. Then the number of nondegenerate roots of
\[
Q_1(x,y(x))=\cdots=Q_n(x,y(x))=0
\]
in $U$ is at most
\[
2^{q(q-1)/2}\,p_1\cdots p_n\,L^q,
\qquad
L=\sum_{r=1}^n (p_r-1)+\min(n,q)M+1.
\]
\end{theorem}

\begin{theorem}[Pfaffian zero-set connected-component bound]\label{thm:GVcc}
Let $X\subseteq\R^n$ be a Pfaffian zero set on a global Pfaff chain on $\R^n$, defined by equations of common chain length $r$, chain degree $\alpha$, and maximal Pfaffian degree $\beta$. Then the number of connected components of $X$ is bounded by
\[
2^{r(r-1)/2+1}\beta(\alpha+2\beta-1)^{n-1}\bigl((2n-1)(\alpha+\beta)-2n+2\bigr)^r.
\]
\end{theorem}

\begin{theorem}[Morse handle attachment]\label{thm:morse-handle}
Let $f:M\to\R$ be a Morse function on a smooth $d$-manifold $M$, and let $a<b$ be real numbers such that $f^{-1}([a,b])$ is compact. Suppose that the interval $(a,b)$ contains exactly one critical value $c$, and that $f^{-1}(c)$ contains critical points $p_1,\dots,p_s$ of Morse indices $\lambda_1,\dots,\lambda_s$. Then, for sufficiently small $\varepsilon>0$, the sublevel set $\{f\le c+\varepsilon\}$ is diffeomorphic, after smoothing corners, to $\{f\le c-\varepsilon\}$ with $s$ handles attached, one $\lambda_i$-handle $\mathbb D^{\lambda_i}\times\mathbb D^{d-\lambda_i}$ attached along $\partial\mathbb D^{\lambda_i}\times\mathbb D^{d-\lambda_i}$ for each critical point $p_i$.
\end{theorem}

\begin{remark}
For $s=1$, Theorem~\ref{thm:morse-handle} is the standard single-handle attachment theorem. The case $s>1$ is the corresponding simultaneous attachment at a common critical level; cf. \citet[Remark~2.9]{Nicolaescu2011}. The superlevel-set statement for a function $h$ follows by applying the sublevel-set statement to $-h$: crossing a critical value of $h$ downward attaches a handle of index $d-\lambda_i$, where $\lambda_i$ is the Morse index of $h$ at the corresponding critical point.
\end{remark}

\section{Proofs and technical results}\label{app:deferred}

\subsection{Proof of Lemma~\ref{lem:low-superlevel-connected}}\label{subsec:low-superlevel-proof}

\begin{proof}
Write $h_i(x)=\alpha_i\phi_i(x)$, so that $h=\sum_i h_i$. Since
\[
\nabla h_i(x)=h_i(x)\Sigma_i^{-1}(\mu_i-x),
\]
we have
\[
x\cdot\nabla h_i(x)=h_i(x)\{x^\top\Sigma_i^{-1}\mu_i-x^\top\Sigma_i^{-1}x\}.
\]
Let $\lambda_i>0$ be the smallest eigenvalue of $\Sigma_i^{-1}$ and let $C_i=\|\Sigma_i^{-1}\mu_i\|$. Then
\[
x^\top\Sigma_i^{-1}\mu_i-x^\top\Sigma_i^{-1}x
\le C_i\|x\|-\lambda_i\|x\|^2<0
\]
for all sufficiently large $\|x\|$. Hence there exists $R>0$ such that
\begin{equation}\label{eq:radial-decrease}
x\cdot\nabla h(x)<0
\qquad (\|x\|\ge R).
\end{equation}
In particular, no critical point lies outside $B_R$.

The compact ball $\overline B_R$ has positive minimum value
\[
c_R=\min_{\overline B_R}h>0.
\]
The critical set is compact by Lemma~\ref{lem:compact-critical}. Since $h$ is Morse, its critical points are isolated; a compact discrete subset of $\R^d$ is finite. Therefore there are only finitely many critical values. Choose $\tau\in(0,c_R)$ smaller than all critical values and also a regular value. Since $h(x)\to0$ as $\|x\|\to\infty$, the superlevel set $S_\tau$ is compact.

Then $\overline B_R\subseteq S_\tau$. If $x\in S_\tau$ and $\|x\|>R$, let $u=x/\|x\|$ and consider $g(t)=h(tu)$ for $t\ge R$. By \eqref{eq:radial-decrease},
\[
g'(t)=u\cdot \nabla h(tu)=\frac{1}{t}(tu)\cdot\nabla h(tu)<0.
\]
Thus $g$ decreases as $t$ increases. Hence the radial segment from $x$ to $Ru$ is contained in $S_\tau$, and $Ru\in \overline B_R\subseteq S_\tau$. Every point of $S_\tau$ is therefore path-connected to $\overline B_R$, and $\overline B_R$ is connected. Hence $S_\tau$ is connected and contains all critical points.
\end{proof}

\subsection{Proof of Lemma~\ref{lem:generic-tilt}}\label{subsec:generic-tilt-proof}

\begin{proof}
Choose pairwise disjoint compact neighborhoods $K_i$ with $p_i\in\mathrm{int}(K_i)$ and
\[
\Phi(p_i)>\max_{x\in\partial K_i}\Phi(x),
\qquad i=1,\dots,M.
\]
Because $e^{c\cdot x}\to 1$ uniformly on the compact set $\bigcup_iK_i$ as $c\to 0$, the strict inequalities
\[
H_c(p_i)>\max_{x\in\partial K_i}H_c(x),
\qquad i=1,\dots,M,
\]
with $H_c(x)=e^{c\cdot x}\Phi(x)$, hold whenever $\|c\|$ is sufficiently small. For such $c$, the maximum of $H_c$ on each compact set $K_i$ is attained in $\mathrm{int}(K_i)$, and therefore gives a local maximum of $H_c$.

If $\Phi$ has finite modal set and $p_1,\dots,p_M$ are its modes, then each $p_i$ is an isolated strict local maximum. Indeed, if $p_i$ were not strict, then points with value $\Phi(p_i)$ would occur arbitrarily near $p_i$; inside a local-maximizing neighborhood of $p_i$, such points would themselves be local maxima, contradicting finiteness of $\Mode(\Phi)$. Thus the preceding construction applies to the list of all modes.

For each component,
\[
e^{c\cdot x}\phi_i(x;\mu_i,\Sigma_i)
=\exp\left(c^\top\mu_i+\frac12c^\top\Sigma_i c\right)
\phi_i(x;\mu_i+\Sigma_i c,\Sigma_i).
\]
Thus $H_c$ is a positive scalar multiple of a $k$-component Gaussian mixture density. If the original mixture is homoscedastic, then the tilted components have the same common covariance. In that case every mean is shifted by the same vector $\Sigma c$, so the affine rank of the means is unchanged.

Finally set
\[
F(x)=-\frac{\nabla\Phi(x)}{\Phi(x)}=-\nabla\log\Phi(x).
\]
Then
\[
\nabla H_c(x)=e^{c\cdot x}\bigl(\nabla\Phi(x)+c\Phi(x)\bigr)
=e^{c\cdot x}\Phi(x)(c-F(x)).
\]
Critical points of $H_c$ are therefore exactly the solutions of $F(x)=c$. Sard's theorem implies that regular values of the smooth map $F:\R^d\to\R^d$ are dense. Choose a nonzero regular value $c$ in the small ball allowed by the preceding paragraph. At a critical point of $H_c$, differentiating the displayed equation for $\nabla H_c$ gives
\[
D^2H_c(x)=e^{c\cdot x}\Phi(x)(-DF(x)),
\]
because the derivative of the scalar factor multiplies the vanishing vector $c-F(x)$. Since $c$ is a regular value, $DF(x)$ is invertible at every solution of $F(x)=c$. Hence every critical point of $H_c$ is nondegenerate. The critical set of $H_c$ agrees with that of the corresponding normalized Gaussian mixture density, and is compact by Lemma~\ref{lem:compact-critical}. Therefore the critical set is finite. The normalized tilted mixture has finite modal set, and each of its local maxima is nondegenerate.
\end{proof}

\subsection{Proof of Theorem~\ref{thm:affine-rank}}\label{subsec:affine-rank-proof}

\begin{proof}
Choose a base point $a\in \aff\{\mu_1,\dots,\mu_k\}$ and let $B=\Sigma^{-1/2}$. Set
\[
\nu_i=B(\mu_i-a),\qquad i=1,\dots,k,
\]
and let $L=\operatorname{span}\{\nu_1,\dots,\nu_k\}$. Since the means have affine rank $r$, the linear space $L$ has dimension $r$. Choose an orthogonal matrix $O$ carrying $L$ onto $\R^r\times\{0\}\subseteq \R^r\times\R^{d-r}$ and define the affine map
\[
T(x)=OB(x-a).
\]
Then $T(\mu_i)=(m_i,0)$ for suitable $m_i\in\R^r$. If $z=(u,v)\in\R^r\times\R^{d-r}$ and $x=T^{-1}(z)$, each pushed-forward component density becomes
\[
\widetilde\phi_i(z)=|\det DT^{-1}|\,\phi_i(T^{-1}(z))
=(2\pi)^{-d/2}\exp\!\left(-\frac12\|z-(m_i,0)\|^2\right).
\]
Therefore
\begin{align*}
\widetilde\Phi(u,v)
&=\sum_{i=1}^k \alpha_i\widetilde\phi_i(u,v) \\
&=(2\pi)^{-d/2}\sum_{i=1}^k \alpha_i\exp\!\left(-\frac12\|u-m_i\|^2\right)e^{-\|v\|^2/2} \\
&=C e^{-\|v\|^2/2}G(u),
\end{align*}
where
\[
G(u)=(2\pi)^{-r/2}\sum_{i=1}^k \alpha_i\exp\!\left(-\frac12\|u-m_i\|^2\right),
\qquad
C=(2\pi)^{-(d-r)/2}.
\]
This proves the asserted factorization.
\end{proof}

\subsection{Proofs for the lower-bound constructions}\label{subsec:lower-bound-proofs}

\begin{proof}[Proof of Proposition~\ref{prop:ambient-monotonicity}]
If $r=d$, there is nothing to prove. Assume $r<d$. Let
\[
G(u)=\sum_{i=1}^k \alpha_i\psi_i(u),\qquad u\in\R^r,
\]
be a $k$-component Gaussian mixture density on $\R^r$, and let
\[
\gamma(v)=\phi(v;0,\Theta),\qquad v\in\R^{d-r},
\]
be any centered Gaussian density on $\R^{d-r}$. Define
\[
\widetilde \Phi(u,v)=G(u)\gamma(v),\qquad (u,v)\in\R^r\times\R^{d-r}.
\]
Writing $\psi_i(u)=\phi(u;\nu_i,\Lambda_i)$, we have
\[
\psi_i(u)\gamma(v)=\phi\bigl((u,v);(\nu_i,0),\diag(\Lambda_i,\Theta)\bigr),
\]
so $\widetilde \Phi$ is a $k$-component Gaussian mixture density on $\R^d$. If $G$ is homoscedastic, choose a fixed $\Theta$ and the lifted mixture is homoscedastic.

We claim that
\[
\Mode(\widetilde \Phi)=\Mode(G)\times\{0\}.
\]
If $u_\ast\in \Mode(G)$, choose a neighborhood $U$ of $u_\ast$ such that $G(u)\le G(u_\ast)$ for all $u\in U$. Since the centered Gaussian $\gamma$ has a unique global maximum at $0$,
\[
\widetilde \Phi(u,v)=G(u)\gamma(v)\le G(u_\ast)\gamma(0)=\widetilde \Phi(u_\ast,0)
\qquad ((u,v)\in U\times\R^{d-r}),
\]
so $(u_\ast,0)\in\Mode(\widetilde \Phi)$.

Conversely, let $(u_\ast,v_\ast)\in\Mode(\widetilde \Phi)$. If $v_\ast\neq 0$, then for some $t\in(0,1)$ sufficiently close to $1$,
\[
\gamma(tv_\ast)>\gamma(v_\ast),
\]
whence
\[
\widetilde \Phi(u_\ast,tv_\ast)=G(u_\ast)\gamma(tv_\ast)>G(u_\ast)\gamma(v_\ast)=\widetilde \Phi(u_\ast,v_\ast),
\]
contradicting local maximality. Thus $v_\ast=0$. If now $u_\ast\notin\Mode(G)$, then every neighborhood of $u_\ast$ contains some $u$ with $G(u)>G(u_\ast)$, and therefore
\[
\widetilde \Phi(u,0)=G(u)\gamma(0)>G(u_\ast)\gamma(0)=\widetilde \Phi(u_\ast,0),
\]
again contradicting local maximality. Hence $u_\ast\in\Mode(G)$, proving the claim.

If $\Mode(G)$ is finite, then the identity $\Mode(\widetilde\Phi)=\Mode(G)\times\{0\}$ shows that $\Mode(\widetilde\Phi)$ is finite with the same cardinality. Therefore every finite modal-set count realized in dimension $r$ is also realized in dimension $d$, and so $m(d,k)\ge m(r,k)$. The same identity also shows that isolated modes of $G$ lift to isolated modes of $\widetilde\Phi$. The homoscedastic statement follows from the homoscedastic construction above.
\end{proof}

\begin{proof}[Proof of Corollary~\ref{cor:lifted-lower}]
Fix $r$ with $2\le r\le \min(d,k)$. By \citet[Theorem~4.1 and its proof]{AmendolaEngstromHaase2019}, there exists a $k$-component Gaussian mixture $\Phi$ on $\R^r$ and at least
\[
L_{\mathrm{AEH}}(r,k)=\binom{k}{r}+k
\]
pairwise disjoint compact neighborhoods satisfying the strict boundary inequalities in Lemma~\ref{lem:generic-tilt}. Applying that lemma and normalizing the tilted density gives a $k$-component Gaussian mixture on $\R^r$ with finite modal set and at least $L_{\mathrm{AEH}}(r,k)$ modes. Proposition~\ref{prop:ambient-monotonicity} then lifts this mixture to dimension $d$. Hence
\[
m(d,k)\ge L_{\mathrm{AEH}}(r,k)=\binom{k}{r}+k
\]
for every such $r$, and taking the maximum over $2\le r\le \min(d,k)$ gives the first claim. The second claim follows because the binomial coefficients are nondecreasing on $0\le r\le \lfloor k/2\rfloor$.
\end{proof}

\begin{proof}[Proof of Proposition~\ref{prop:padding-products}]
For part~(a), let the original mixture $\Phi$ have isolated modes $p_1,\dots,p_M$. Since each $p_i$ is an isolated local maximum, it is strict. If $\ell=0$, Lemma~\ref{lem:generic-tilt}, followed by normalization, gives a mixture with a compact nondegenerate critical set, hence finite modal set, and at least these $M$ nondegenerate modes. Assume from now on that at least one component is to be added. It is enough to add one component; adding several components is obtained by choosing the new components far from the original mixture and far from each other, and repeating the same argument on disjoint balls. Choose pairwise disjoint closed balls
\[
B_i=\overline B(p_i,\rho_i),\qquad i=1,\dots,M,
\]
such that
\[
\sup_{x\in \partial B_i}\Phi(x)<\Phi(p_i)
\qquad (i=1,\dots,M).
\]
Set
\[
\eta=\min_{1\le i\le M}\left(\Phi(p_i)-\sup_{x\in\partial B_i}\Phi(x)\right)>0
\]
and $K=\bigcup_iB_i$.

Choose any positive definite covariance matrix $\Theta$; in the homoscedastic case choose $\Theta$ equal to the common covariance of the original mixture. Let $\psi_a(x)=\phi(x;a,\Theta)$. Fix a radius $\rho>0$ such that
\[
\sup_{\|x-a\|=\rho}\psi_a(x)<\frac12\psi_a(a);
\]
this inequality is independent of $a$ after translation. Choose $\theta\in(0,1/2]$ and then choose
\[
\delta<\frac{\theta\psi_a(a)}{2(1-\theta)},
\]
where $\psi_a(a)$ is independent of $a$. Because $\Phi$ and all Gaussian tails tend to zero at infinity, we can choose $a$ so far from $K$ that $C=\overline B(a,\rho)$ is disjoint from $K$,
\[
\sup_{x\in K}\psi_a(x)<\frac{\eta}{8},
\qquad
\sup_{x\in C}\Phi(x)<\delta.
\]
Define
\[
\widetilde\Phi=(1-\theta)\Phi+\theta\psi_a.
\]
For $x\in\partial B_i$,
\[
\widetilde\Phi(x)\le (1-\theta)\sup_{\partial B_i}\Phi+\theta\sup_K\psi_a,
\]
whereas $\widetilde\Phi(p_i)\ge (1-\theta)\Phi(p_i)$. Hence
\[
\widetilde\Phi(p_i)-\sup_{\partial B_i}\widetilde\Phi
\ge (1-\theta)\eta-\theta\frac{\eta}{8}>0.
\]
Thus each $B_i$ contains a mode of $\widetilde\Phi$.

On the new ball $C$,
\[
\widetilde\Phi(a)\ge \theta\psi_a(a),
\]
whereas
\[
\sup_{\partial C}\widetilde\Phi
\le (1-\theta)\delta+\frac{\theta}{2}\psi_a(a).
\]
By the choice of $\delta$, $\widetilde\Phi(a)>\sup_{\partial C}\widetilde\Phi$, so $\widetilde\Phi$ attains a local maximum inside $C$. This new mode is distinct from the modes inside the balls $B_i$. The balls $B_1,\dots,B_M,C$ therefore give $M+1$ pairwise disjoint compact neighborhoods satisfying the strict boundary inequalities in Lemma~\ref{lem:generic-tilt}. Applying that lemma to these neighborhoods and normalizing the tilted function gives a Gaussian mixture density with all critical points nondegenerate and with at least $M+1$ nondegenerate modes. Its critical set is compact by Lemma~\ref{lem:compact-critical}; since all critical points are nondegenerate, the critical set is finite, and hence the modal set is finite. Thus one component can be added while producing a finite-modal-set mixture and increasing the number of nondegenerate modes by at least one. For arbitrary $\ell$, repeat this construction inductively: after each addition, retain disjoint balls witnessing the nondegenerate modes already obtained, and choose the next center outside a sufficiently large compact set containing all previous witness balls. The same boundary inequalities persist after taking the new component sufficiently remote and its weight sufficiently small. The homoscedastic statement follows because each new component may be chosen with the common covariance and the final tilt preserves homoscedasticity.

For part~(b), write
\[
F(x)=\sum_{i=1}^{k_1}\alpha_i\phi_i(x),
\qquad
G(y)=\sum_{j=1}^{k_2}\beta_j\psi_j(y),
\]
and define
\[
H(x,y)=F(x)G(y).
\]
Then
\[
H(x,y)=\sum_{i=1}^{k_1}\sum_{j=1}^{k_2}\alpha_i\beta_j\,\phi_i(x)\psi_j(y),
\]
and each product $\phi_i(x)\psi_j(y)$ is a Gaussian density on $\R^{d_1+d_2}$ with block-diagonal covariance, so $H$ is a $k_1k_2$-component Gaussian mixture density. If $F$ and $G$ are homoscedastic, all block-diagonal covariances are the same across $(i,j)$, so $H$ is homoscedastic.

We claim that
\[
\Mode(H)=\Mode(F)\times\Mode(G).
\]
If $x_\ast\in\Mode(F)$ and $y_\ast\in\Mode(G)$, choose neighborhoods $U$ of $x_\ast$ and $V$ of $y_\ast$ such that $F(x)\le F(x_\ast)$ on $U$ and $G(y)\le G(y_\ast)$ on $V$. Then
\[
H(x,y)=F(x)G(y)\le F(x_\ast)G(y_\ast)=H(x_\ast,y_\ast)
\qquad ((x,y)\in U\times V),
\]
so $(x_\ast,y_\ast)\in\Mode(H)$.

Conversely, if $(x_\ast,y_\ast)\in\Mode(H)$ and $x_\ast\notin\Mode(F)$, choose $x$ arbitrarily close to $x_\ast$ with $F(x)>F(x_\ast)$. Since $G(y_\ast)>0$,
\[
H(x,y_\ast)=F(x)G(y_\ast)>F(x_\ast)G(y_\ast)=H(x_\ast,y_\ast),
\]
contradicting local maximality. Thus $x_\ast\in\Mode(F)$, and similarly $y_\ast\in\Mode(G)$. This proves the claim. In particular, the Cartesian products of isolated modes are isolated modes, and at least $M_1M_2$ isolated modes are obtained. Isolated local maxima are strict, and products of strict local maxima of positive functions are strict local maxima. Applying Lemma~\ref{lem:generic-tilt} to these finitely many strict product modes and normalizing the tilted function gives a mixture with all critical points nondegenerate and at least $M_1M_2$ nondegenerate modes. By Lemma~\ref{lem:compact-critical}, its critical set is compact; since all critical points are nondegenerate, the critical set is finite, and hence the modal set is finite.
\end{proof}

\begin{proof}[Proof of Theorem~\ref{thm:padding-product-lower}]
Fix $s$ with $1\le s\le \min(d,\lfloor \log_2 k\rfloor)$, and write
\[
q=\left\lfloor\frac{d}{s}\right\rfloor,
\qquad
r=d\bmod s.
\]
Then $d=qs+r$ and $0\le r<s$. By \citet{RayRen2012}, for each integer $q'\ge 1$ there exists a two-component Gaussian mixture in $\R^{q'}$ with exactly $q'+1$ modes. Take $s-r$ copies in dimension $q$ and $r$ copies in dimension $q+1$. By part~(b) of Proposition~\ref{prop:padding-products}, their Cartesian product is a $2^s$-component Gaussian mixture in
\[
q(s-r)+(q+1)r=d
\]
dimensions with at least
\[
(q+1)^{s-r}(q+2)^r
\]
modes.

If $k=2^s$, this proves the displayed bound. If $k>2^s$, apply part~(a) of Proposition~\ref{prop:padding-products} with $\ell=k-2^s$ to add $k-2^s$ further components and at least $k-2^s$ further modes. This yields
\[
m(d,k)\ge k-2^s+(q+1)^{s-r}(q+2)^r
= k-2^s+\left(\left\lfloor\frac{d}{s}\right\rfloor+1\right)^{s-(d\bmod s)}
\left(\left\lfloor\frac{d}{s}\right\rfloor+2\right)^{d\bmod s}.
\]
This proves the stated bound for the chosen value of $s$.
\end{proof}

\subsection{Proof of Theorem~\ref{thm:simplex-seed}}\label{subsec:simplex-seed-proof}

\begin{proof}
Let $n=K-1\ge2$. Choose regular simplex vertices $v_1,\dots,v_K\in\R^n$ with
\[
\sum_{i=1}^K v_i=0,
\qquad
\|v_i\|=1,
\qquad
v_i\cdot v_j=-\frac1n\quad(i\ne j).
\]
For $\tau>0$, ignore the common normalizing constant and consider
\[
F_\tau(x)=\sum_{i=1}^K \exp\left(-\frac{\|x-v_i\|^2}{2\tau}\right).
\]
This is a positive scalar multiple of an equal-weight homoscedastic Gaussian mixture with common covariance $\tau I_n$.

First, $0$ is critical by symmetry. Since
\[
\sum_{i=1}^K v_iv_i^\top=\frac{K}{n}I_n,
\]
we get
\[
D^2F_\tau(0)=e^{-1/(2\tau)}
\left(\frac1{\tau^2}\sum_{i=1}^K v_iv_i^\top-\frac{K}{\tau}I_n\right)
=K e^{-1/(2\tau)}\left(\frac{1}{n\tau^2}-\frac1\tau\right)I_n.
\]
Thus $0$ is a strict local maximum whenever $\tau>1/n$.

We now show that, for $\tau>1/n$ sufficiently close to $1/n$, there is also a strict local maximum on each ray from the origin through a simplex vertex. It suffices to study $x=t v_1$, $0<t<1$; the same calculation applies after relabeling the vertices. Put
\[
A(t)=\exp\left(-\frac{(t-1)^2}{2\tau}\right),
\qquad
B(t)=\exp\left(-\frac{t^2+1+2t/n}{2\tau}\right).
\]
Then
\[
F_\tau(tv_1)=A(t)+nB(t).
\]
At $x=tv_1$, the first summand has weight $A(t)$ and the remaining $n$ summands have common weight $B(t)$. Since $\sum_{i=2}^K v_i=-v_1$,
\[
\nabla F_\tau(tv_1)=\frac1\tau\left(A(t)(v_1-tv_1)+B(t)\sum_{i=2}^K(v_i-tv_1)\right)
=\frac1\tau\left((1-t)A(t)-(1+nt)B(t)\right)v_1.
\]
Thus $t v_1$ is critical if and only if the scalar factor in this display vanishes. Equivalently, differentiating along the line $\R v_1$ gives
\[
\frac{\mathrm d}{\mathrm dt}F_\tau(tv_1)=\frac1\tau\left((1-t)A(t)-(1+nt)B(t)\right).
\]
Moreover,
\[
\frac{A(t)}{B(t)}=\exp\left(\frac{Kt}{n\tau}\right).
\]
Thus the critical equation is
\[
(1-t)e^{at}=1+nt,
\qquad
 a=\frac{K}{n\tau}.
\]
Let
\[
\ell_a(t)=\log(1-t)+at-\log(1+nt),
\qquad 0<t<1.
\]
The critical equation is $\ell_a(t)=0$. At the threshold $\tau=1/n$, one has $a=K=n+1$, and the Taylor expansion at $0$ is
\[
\ell_K(t)=\frac{n^2-1}{2}t^2+O(t^3).
\]
The coefficient is positive because $n\ge2$, so there is a small $t_0>0$ with $\ell_K(t_0)>0$. By continuity, $\ell_a(t_0)>0$ for all $a<K$ sufficiently close to $K$, equivalently for all $\tau>1/n$ sufficiently close to $1/n$. For such $a$,
\[
\ell_a'(0)=a-K<0,
\]
so $\ell_a(t)<0$ for all sufficiently small positive $t$. Also $\ell_a(t)\to-\infty$ as $t\uparrow1$. Hence $\ell_a$ has at least two zeros in $(0,1)$. Let $t_\ast$ be the right endpoint of a component of $\{t:\ell_a(t)>0\}$. Then $\ell_a(t_\ast)=0$. Since $\ell_a$ is real analytic and is not identically zero, its zeros are isolated. Thus, after restricting to a sufficiently small neighborhood of $t_\ast$, one has $\ell_a(t)>0$ for $t<t_\ast$ close to $t_\ast$ and $\ell_a(t)<0$ for $t>t_\ast$ close to $t_\ast$. The derivative of the one-variable function $t\mapsto F_\tau(tv_1)$ has the same sign as $\ell_a(t)$, so this restriction has a strict local maximum at $t_\ast v_1$.

It remains to check directions orthogonal to this line. Let $w$ be a unit vector with $w\cdot v_1=0$. For a single summand $E_i(x)=\exp(-\|x-v_i\|^2/(2\tau))$,
\[
D^2E_i(x)[w,w]
=E_i(x)\left(\frac{((v_i-x)\cdot w)^2}{\tau^2}-\frac{1}{\tau}\right).
\]
At $x=tv_1$, the $i=1$ term contributes $-A(t)/\tau$. For the remaining vertices, $E_i(tv_1)=B(t)$ and
\[
\sum_{i=2}^K (v_i\cdot w)^2
=w^\top\left(\sum_{i=2}^K v_iv_i^\top\right)w
=w^\top\left(\frac{K}{n}I_n-v_1v_1^\top\right)w
=\frac{K}{n}.
\]
Therefore
\[
D^2F_\tau(tv_1)[w,w]
=-\frac{A(t)}{\tau}+B(t)\left(\frac{K}{n\tau^2}-\frac{n}{\tau}\right).
\]
At the critical value $t=t_\ast$, the relation $A(t_\ast)/B(t_\ast)=(1+nt_\ast)/(1-t_\ast)$ gives
\[
D^2F_\tau(t_\ast v_1)[w,w]
=\frac{B(t_\ast)}{\tau}
\left(-\frac{1+nt_\ast}{1-t_\ast}+\frac{K}{n\tau}-n\right).
\]
Since $\tau>1/n$, we have $K/(n\tau)<K=n+1$, while $(1+nt_\ast)/(1-t_\ast)>1$. The expression in parentheses is therefore strictly less than
\[
-1+(n+1)-n=0.
\]
Thus the second derivative is negative in every unit direction orthogonal to $v_1$.

Finally, for every $t$ near $t_\ast$ and every vector $w$ satisfying $w\cdot v_1=0$,
\[
DF_\tau(tv_1)[w]
=\frac{1}{\tau}\sum_{i=1}^K E_i(tv_1)(v_i-tv_1)\cdot w
=\frac{B(t)}{\tau}\left(\sum_{i=2}^K v_i\right)\cdot w=0,
\]
because $\sum_{i=2}^K v_i=-v_1$. By continuity, the negative definiteness in the directions orthogonal to $v_1$ persists for $t$ in a small interval around $t_\ast$. Hence, for such $t$, the restriction of $F_\tau$ to the affine space $tv_1+\{w:w\cdot v_1=0\}$ has a strict local maximum at $tv_1$. Combining this with the strict local maximum of the one-dimensional restriction $t\mapsto F_\tau(tv_1)$ proves that $t_\ast v_1$ is a strict local maximum of $F_\tau$. By symmetry, $t_\ast v_i$, $i=1,\dots,K$, are $K$ distinct strict local maxima. Together with the central maximum at $0$, this gives at least $K+1$ strict local maxima. Choose pairwise disjoint compact neighborhoods witnessing these strict local maxima. Applying Lemma~\ref{lem:generic-tilt} to those neighborhoods and normalizing the tilted function gives a homoscedastic $K$-component Gaussian mixture with finite modal set and at least $K+1$ modes.
\end{proof}

\end{document}